%% file: neurips_2024.tex
\documentclass{article}

\pdfoutput=1

\usepackage[preprint,nonatbib]{neurips_2024}

\usepackage[utf8]{inputenc} 
\usepackage[T1]{fontenc}    
\usepackage{hyperref}       
\usepackage{url}            
\usepackage{booktabs}       
\usepackage{amsfonts}       
\usepackage{nicefrac}       
\usepackage{microtype}      
\usepackage{xcolor}         

\usepackage[backend=biber,style=numeric]{biblatex}
\usepackage{ant_cmds2}

\title{Further Understanding of a Local Gaussian Process Approximation: Characterising Convergence in the Finite Regime}

\author{%
Anthony Stephenson \\
Department of Mathematics \\ 
University of Bristol \\
\texttt{ant.stephenson@bristol.ac.uk} \\
\And
Robert Allison \\
Dpeartment of Mathematics \\
University of Bristol \\
\texttt{marfa@bristol.ac.uk} \\
\And
Edward Pyzer-Knapp \\
IBM Research \\
}

\begin{document}

\maketitle

\begin{abstract}
We show that common choices of kernel functions for a highly accurate and massively scalable nearest-neighbour based GP regression model (GPnn: \cite{GPnn}) exhibit gradual convergence 
to asymptotic behaviour as dataset-size \(n\) increases. For isotropic kernels such as Mat\'{e}rn and squared-exponential, an upper bound on the predictive MSE can be obtained as \(\bigO{n^{-\frac{p}{d}}}\) for input dimension \(d\), \(p\) dictated by the kernel (and \(d>p\)) and fixed number of nearest-neighbours \(m\) with minimal assumptions on the input distribution. Similar bounds can be found under model misspecification and combined to give overall rates of convergence of both MSE and an important calibration metric. We show that lower bounds on \(n\) can be given in terms of \(m\), \(l\), \(p\), \(d\), a tolerance \(\varepsilon\) and a probability \(\delta\). When \(m\) is chosen to be \(\bigO{n^{\frac{p}{p+d}}}\) minimax optimal rates of convergence are attained.
Finally, we demonstrate empirical performance and show that in many cases convergence occurs faster than the upper bounds given here.
\end{abstract}

\input{content}

\printbibliography 

\clearpage
\input{appendix}
\end{document}

%% file: content.tex
\section{Introduction}
\label{sec:intro}

\paragraph{Gaussian Processes (GPs)} are stochastic processes which can be viewed as distributions over functions, widely used for non-linear regression with a focus on well-calibrated uncertainty. Despite a rich history and broad applications, exact-GP regression's \(\bigO{n^3}\) scaling hinders its practicality for moderately large datasets. Numerous approximation techniques aim to address this and capture the headline attributes of uncertainty quantification and model flexibility. For a thorough review of such approximations, we refer the reader to \cite{GP_approx_review}. 

Within the ML community, variational sparse approximations are amongst the most popular, with desirable theoretical properties and effective empirical performance. Often, empirical Bayes procedures are preferred, estimating kernel hyperparameters by maximum likelihood estimation or maximisation of the evidence lower bound (ELBO) in variational settings. In contrast, the geospatial research community makes extensive use of local approximations and often favours fully Bayesian methodologies that place priors over kernel hyperparameters and rely on sampling-based approaches for inference. Both strategies have had significant success in improving scalability of GP regression. Here, we focus on empirical Bayes local approximation methods to drive the scalability of GPs further still.

\paragraph{Nearest Neighbour GP Approximations} have a long history in the geo-spatial research community \cite{vecchia_estimation_1988,Stein2002,Stein2004,Stein2010,Gramacy2015,Datta2016,finley_efficient_2018}. More recently, the ML community has applied variational methods to the problem in \cite{Tran2020,Wu2022}. We focus on a straightforward predictive approximation, GPnn, introduced in \cite{GPnn}. In that work, the authors decouple parameter estimation from prediction to achieve a highly performant and scalable model by choosing a cheap training procedure (using latched subsets to estimate parameters) and feeding the output into an independent predictive mechanism. They show that, asymptotically, metrics measuring predictive performance (MSE, NLL, Calibration) become insensitive to kernel hyperparameters and even kernel choice (under relatively weak assumptions of weak stationarity of the generative random field). Their work demonstrates theoretical convergence to approximately optimal performance (up to terms reciprocal in the number of neighbours) and includes supporting empirical evidence, alongside strong performance on a selection of UCI datasets: outperforming benchmark distributed GP approximations and SVGP almost uniformly. A more thorough description of existing nearest-neighbour methods is given in \ref{appx:nngps} and \cite{GPnn}.

\paragraph{Convergence Rates} provide a substantive extension to asymptotic convergence results, which offer important assurances of model behaviour (in the limit) and suggestions of trend in performance in terms of \(n\). Rates can provide a guide as to when limiting behaviour might become relevant and the suddenness by which this occurs. This is crucial to understanding the potential performance of a model and can (in principle) be compared to convergence rates of other models to allow a practitioner to select an approach most likely to achieve their aims. With this in mind, we review existing relevant convergence rates found in the literature on both classical nearest-neighbour (\(k\)-NN) methods as well as contraction rates of both exact and approximate GPs.

\paragraph{Structure and Contributions:}
In this work we extend the theoretical results given in \cite{GPnn} from purely asymptotic (supported by empirical performance) to rates of convergence. This provides practitioners with the capability to quantify the performance and robustness of the model under misspecification in terms of dataset size \(n\). To do this, we derive rates of convergence of both the MSE and a measure of (weak) calibration in terms of \(n\), as well as describing important phenomenology regarding interactions of lengthscale, dimension and \(n\), both qualitatively and empirically.

\Cref{sec:background} provides a brief background on exact-GP regression as usually applied in the ML community, a description of the model known as GPnn (as given in \cite{GPnn}) and relevant convergence rates from the literature on related models. 
\Cref{sec:rates} gives the main results of this manuscript: the convergence rates for the GPnn predictions under various forms of model misspecification with respect to the underlying data-generating process. \Cref{sec:exp} gives empirical evidence of the convergence behaviour observed under simulated conditions and a qualitative discussion on its relation to the theory given in the preceding section. Finally, a brief discussion of the theoretical and empirical results is offered in \Cref{sec:discussion}, followed by a summary of outstanding research directions, some of which are actively pursued by the authors.

\section{Background}
\label{sec:background}

\subsection{Gaussian Processes}
\label{sec:GP}

Exact GP regression, when applied in an empirical Bayes framework, is carried out by defining a GP prior over a latent function \(f: \mathbb{R}^d \rightarrow \mathbb{R}\) with assumptions on the properties of \(f\) encoded by the choice of prior positive definite covariance function $c(\cdot,\cdot):\mathbb{R}^d \times \mathbb{R}^d \rightarrow \mathbb{R}_{+}$. A model of the form \(y(\xv)=f(\xv)+\xi\) is assumed with \(\xi\sim \mc{N}(0,\sigma_\xi^2)\) where \(\sigma_\xi^2\) is termed the noise-variance and forms one of the elements of a vector \(\bm{\theta}\) of hyperparameters that, in this methodology, are estimated by minimising the marginal (Gaussian) negative loglikelihood: \(-\log p(\bm{y}|X,\bm{\theta}) = \frac{1}{2}\{ \bm{y}^T K_{\bm{\theta}}^{-1} \bm{y} + \log|K_{\bm{\theta}}| + n \log(2 \pi) \}\). Estimates of these hyperparameters, \(\thv\), are then fed into the following expressions for the mean and variance of the predictive distribution (obtained via Bayes rule) at a particular test point \(\xv^*\), 
\begin{align}
  y^*\given X,\bm{y} &\sim \mathcal{N}(\mu^*,\predvar) \label{eq:gp_pred}\\
  \mu^* &= {\kstar}^T K ^{-1}\bm{y} \label{eq:gp_mean_pred}\\
  \predvar &= \sigma_f^2 - {\kstar}^T K^{-1}\kstar + \sigma_\xi^2 \label{eq:gp_var_pred}
\end{align}
where $K = K_{\hat{\bm{\theta}}}$ with components \([K]_{ij}=k_{\hat{\bm{\theta}}}(\xv_i,\xv_j)\); the vector $\kstar$ has components $k^*_i= k_{\hat{\bm{\theta}}}(\xv_i,\xv^*)$, and \(k_{\bm{\theta}}(\xv,\xv') = \sigma_f^2c(\xv/l,\xv'/l) + \delta_{\xv,\xv'}\sigma_\xi^2\) (when we take the special case $\bm{\theta} = (l,\sigma_\xi^2,\sigma_f^2)$) with a ``normalised'' covariance function \(c(\cdot,\cdot)\) such that \(c(\xv,\xv)=1\).
See \cite{GP_book} for a classic text on the subject.

\subsection{GPnn}
\label{sec:GPnn}

GPnn predictions are made using the same formulae, \eqref{eq:gp_pred}, \eqref{eq:gp_mean_pred} and \eqref{eq:gp_var_pred}, but where conditioning is now with respect to the nearest neighbour sets \(N \in \mathbb{R}^{m \times d}\) and \(\bm{y}_N \in \mathbb{R}^m\):
\begin{align}
  y^*\given N(\bm{x}^*),\bm{y}_N &\sim \mc{N}(\mu_N^*,\predvar_N)\label{eq:gpnn_pred} \\
  \mu_N^* &= \kstar[]_N^T {K}_N^{-1}\bm{y}_N \label{eq:gpnn_mean_pred}\\
  \predvar_N &= {\sigma}_f^2 - \underbrace{\kstar[]_N^T {K}_N^{-1}\kstar[]_N}_{\kappa_N} + {\sigma}_\xi^2 .\label{eq:gpnn_var_pred}
\end{align}

\subsection{Convergence Rates in \texorpdfstring{\(k\)}{k}-NN Approximations}
\label{sec:rates_longaxis}


Derivations of convergence rates for the classical \(k\)-NN algorithm are relatively well known, being computed in e.g. \cite{Gyorfi2010} as \(\mc{O}\big(n^{-\frac{2}{d+2}}\big)\) under a Lipschitz condition on the true function and a bounded input space. In addition to this reference, we draw particular attention to \cite{Kohler2006} which extends classical \(k_n\)-NN convergence rates from bounded input spaces to the unbounded case (provided a moment condition on the input distribution is satisfied) to give a rate of \(\mc{O}\big(n^{-\frac{2r}{2r+d}}\big)\), where \(0<r\leq 1\) is the H\"{o}lder exponent of the true function. 
Note that, in each case, the number of nearest neighbours \(k_n\) is taken to grow with \(n\) to ensure the estimators are consistent.

\subsection{Convergence Rates of GP Posteriors}
\label{sec:rates_gp}
To provide context for the rates we will derive, it is helpful to first review existing convergence rates in the literature for exact GP regression as well as common approximation methods (such as SVGP). For the former, we look to \cite{van_der_vaart_information_2011} and highlight the results the authors obtain in Theorem 5, which
gives rates of contraction of the GP posterior of \(\mc{O}\big(n^{-2\min(\alpha,\beta)/(2\alpha+d)}\big)\) for a kernel \(k_s\) on \([0,1]^d\) 
which defines an RKHS equipped with a norm equivalent to the Sobolev space \(W^s_2([0,1]^d)\) 
where \(s:=\alpha+d/2\), \(\alpha>0\) and \(\min(\alpha,\beta)>d/2\) and \(\beta\) controls the smoothness of the true function.

For variational methods, \cite{burt_convergence_2020} give rates at which the approximate posterior converges to the exact GP posterior which are further refined in \cite{vakili_improved_2022}. Subsequently, \cite{Nieman2022} derive rates under conditions directly analagous to those in \cite{van_der_vaart_information_2011} and find that these rates match those of exact GP regression when the number of inducing points \(m\) is chosen appropriately. For example, for the Mat\'{e}rn-\(\nu\), the authors set \(m=\tilde{\mc{O}}\big(n^{\frac{d}{2\nu+d}}\big)\) to give a rate of \(\ordtilde{n^{-\frac{\nu}{2\nu+d}}}\). 

We explore the relationships between the function spaces described in this section and above in more detail in \ref{sec:function_spaces}.


\section{Convergence Rates of GPnn}
\label{sec:rates}

In this section we derive convergence rates for GPnn. The calculations rely primarily on assumptions on the form of the kernel and its relation to Euclidean distance. In particular, we assume some form of Lipschitz-type relation holds (\ref{ass:kernel_dist}) so that we can make use of classical results on convergence rates of nearest-neighbours under Euclidean distances. In the well-specified model scenario we provide bounds for the induced metrics of some non-trivial kernel choices and this, with a mild technical condition, is sufficient to derive the rates given. For the misspecified case, a more involved approach is required, as the former fails to give useful results. In this case we use Neumann expansions of the kernel inverse and combine these with similar bounds on the induced metrics of the model kernel. The resulting computations require more stringent conditions to hold than in the well-specified case. 


As a first step (after defining relevant notation), we give some definitions and assumptions that will be used extensively in the derivation of the results.

\paragraph{Notation:} 
\label{sec:notation}
For simplicity we adopt notation as used in \cite{GPnn} with \(\rho(\xv,\xv')=\sigma_f\sqrt{1-c(\xv/l,\xv'/l)}\) being the kernel-induced distance function over \(\mathbb{R}^d\) \cite{Scholkopf2001}; defining \(\xrv_{(j,n)}(\xv^*)\) as the \(j^{th}\) nearest neighbour random variable to a test point \(\xv^*\) under \(\rho\), which we abbreviate to \(\xrv_{(j)}\) when the context is clear, and \(\xv_{(j)}(\xv^*) \in N_m(\xv^*)\) as the realised \(j^{th}\) nearest neighbour of the test point \(\xv^*\) from a training set \(X\) in the set of \(m\) nearest neighbours \(N_m\). We correspondingly define \(\epsilon_i = \rho^2(\xrv_{(i)},\xv^*)\) and \(\epsilon_{ij} = \rho^2(\xrv_{(i)},\xrv_{(j)})\).
We will on occasion abuse notation and write any of \(k,c,\rho\) as functions of a single argument, when the underlying kernel function (\(c\)) is isotropic and hence a function of \(\Delta := \norm{\xv-\xv'}\) (or \(\norm{\xrv-\xrv'}\) when the context is clear that we are dealing with RVs). Similarly, we define \(\Delta_j^{\xv}:=\norm{\xrv_{(j,n)}(\xv)-\xv}\). 
Additionally, we define the quantities \(\eps^\bullet_{\min},\eps^{\bullet\bullet}_{\min}\) by \(\min_{i}\eps_{i},\min_{ij}\eps_{ij}\) respectively, and similarly for \(\eps_{\max}\). 
We use hatted notation to indicate any form of misspecification (clarified in \cref{sec:mis_rate}).
We use the notation \(A \preceq B\) to mean the matrix \(B-A\) is positive semi-definite and \(a_n \lesssim b_n\) to mean \(a_n \leq b_n + o(b_n)\), i.e. ignoring negligible higher order terms. We will generally truncate expansions in \(\eps_i,\eps_{ij}\) etc. to second order and, when using \(\bigO{\cdot}\) notation, we take \(\eps\) to represent all definitions. 
We define the performance metrics we analyse as \(f_n^{\MSE}(\thv) = \Eyy{(y^*-\mu_N^*)^2}\), \(f_n^{\CAL}(\thv) = f_n^{\MSE}(\thv)/\predvar_N\) and \(f_n^{\NLL}(\thv)=\half(\log\predvar_N + \log2\pi) + f_n^{\CAL}(\thv)\).

\begin{definition}[Support]\label{def:support}
    Let \(P_\xrv\) be the probability measure of \(\xrv\) and \(S^\rho_{\xv,\epsilon}\) the closed ball of radius \(\epsilon>0\) under the metric \(\rho\) centred at \(\xv\). Then we define \(\support(P_\xrv) = \{\xv : P_\xrv(S^\rho_{\xv,\epsilon})>0\, \forall\,\epsilon >0\}\).
\end{definition}

\begin{definition}[Weakly-faithful]\label{def:weakly_faithful}
    We define a pair of metrics \(\rho(\cdot,\cdot),\hat{\rho}(\cdot,\cdot)\) to be \emph{weakly-faithful} over \(\mc{X}\) w.r.t. each other if the following condition holds: 
    \(\rho(\xrv_{(m)}(\xv),\xv)\xrightarrow{n\rightarrow\infty}0 \iff \hat{\rho}(\xrv_{(m)}(\xv),\xv)\xrightarrow{n\rightarrow\infty}0\) for all \(\xv \in \mc{X}\).
\end{definition}

\begin{assumption}
\label{ass:convergence}
Below, we list the main assumptions that we will utilise piecewise throughout the following sections apart from \ref{ass:wsrf} which is applied everywhere.
  \begin{enumerate}[start=0,before=\leavevmode,label=assumption,ref=(A\arabic*)]
    \item \label[assumption]{ass:wsrf} \(y_i=f(\bm{x}_i) + \xi_i\) with \(\xi_i\overset{\textrm{iid}}{\sim}P_\xi\); \(f(\bm{x})\sim \mc{WSRF}(\sigma_f^2 c(./l,./l))\), \(y_i\given f(\bm{x}_i) \sim P_\xi\) and \(\E[\xi]=0\), \(\E[\xi^2]=\sigma_\xi^2\); with \(\mc{WSRF}\) a Weakly Stationary Random Field.
    \item \label[assumption]{ass:kernel_dist} There exist constants \(L \in \mathbb{R}_+,p \in (0,2]\)  such that \(\rho^2(\xv,\xv';l) < L\left(\frac{\norm{\xv-\xv'}}{l}\right)^p\). 
    \item \(\xrv\overset{\textrm{iid}}{\sim}P_\xrv\) and \(\xrv^* \in \support(P_\xrv)\) under the generative metric defined by \(c(\cdot,\cdot)\); there exists a \(\beta > p\) such that \(\E{\norm{\xrv}^\beta}<\infty\) and \(d>p\) (where \(p\) is defined by \ref{ass:kernel_dist}). \label[assumption]{ass:X}
    \item \label[assumption]{ass:faithful} \(c(\cdot,\cdot),\hat{c}(\cdot,\cdot)\) are stationary kernels whose induced distance functions are \emph{weakly faithful} metrics (\autoref{def:weakly_faithful}). 
    \item \label[assumption]{ass:eps_min} \(\eps^{\bullet\bullet}_{\min} < \frac{m\sigma_f^2 + \sigma_\xi^2}{m-1}\). 
  \end{enumerate}
\end{assumption}

 Whilst the first part of \ref{ass:X} is standard in the literature, the moment condition allows us to forgo the stronger condition of requiring bounded support. \ref{ass:kernel_dist} holds for many commonly used kernels, including the squared-exponential, Mat\'{e}rn kernels with \(\nu=\half\) and \(\nu\neq1\) (which includes the squared-exponential with \(\nu=\infty\)) and the rational quadratic kernel (see \ref{appx:assumptions}). Note also that \ref{ass:kernel_dist} implies (by Kolmogorov's continuity theorem) a local H\"{o}lder continuity assumption, which is weaker than a general H\"{o}lder continuous assumption (see \ref{sec:function_spaces}).
 \ref{ass:faithful} is a weak assumption that likewise we would expect to hold for many oft-used kernels, especially all induced distance metrics that are monotonic functions of Euclidean distance (such as the Mat\'{e}rn). \ref{ass:eps_min} in contrast is a technical condition required for the machinery used subsequently to be valid. Fortunately, the condition given for \(\eps^{\bullet\bullet}_{\min}\) is not very stringent: Suppose we normalise the observation variance so that \(\sigma_f^2+\sigma_\xi^2=1\), then the condition can be translated to \(\eps^{\bullet\bullet}_{\min}<\sigma_f^2+\frac{1}{m-1}\), which is always satisfied.

\subsection{Under the Generative Model}
\label{sec:gen_rate}
Although we will go on to derive a rate under a misspecified model, we first describe the simpler case when the model is well-specified to allow a gradual exposition as the complexity of the situation increases. 
Here, when we use the term ``well-specified'', we mean in terms of the second order structure of the data generating process. This allows us to retain a degree of ``misspecification'' with respect to the noise process (which is permitted to be non-Gaussian, provided its moments satisfy the usual constraints) and the stochastic process which is understood to be a WSRF, {\it not necessarily} a GP, \ref{ass:wsrf}.

First we give a bound on the MSE as a function of the nearest neighbour distances (under the kernel induced metric):

\begin{lemma}[Matched parameter upper bound]
    \label{lem:MSE_UB}
    With \(\hat{c}(\cdot)=c(\cdot)\) and \(\thv=\bm{\theta}\), i.e. a well-specified model,
    \begin{equation*}
        f_n^{\MSE} \leq \, \sigma_\xi^2 + \sigma_f^2 - \frac{\sigma_f^2 - 2\avg{\eps_{\bullet}}}{1+\sigma_\xi^2/m\sigma_f^2} + \bigO{\eps^2}.
    \end{equation*}
\end{lemma}
\begin{proof}
    We combine the definition \(f_n^{\MSE} = \sigma_f^2 + \sigma_\xi^2 - \kappa_N\) with \cref{lem:kappa} to obtain the result. (\(\kappa_N=\kstar_N^TK_N^{-1}\kstar_N\)).
\end{proof}

To convert \cref{lem:MSE_UB} to a result explicitly depending on \(n\), we adopt proof techniques used for similar results in the nearest-neighbour literature (e.g. \cite{Kohler2006}). The following lemma is the key result that is used in both the well-specified and misspecified theorems to come.

\begin{lemma}
    \label{lem:eps_i}
    Assume \ref{ass:kernel_dist} is satisfied for all \(\xv,\xv'\) and that \ref{ass:X} holds, then
    \eq{
        \Ex{\avg{\eps_\bullet}} &\leq C_{\bm{\theta},d}\left(\frac{m}{n}\right)^\frac{p}{d}
    }
    where \(\avg{\eps_\bullet} = \frac{1}{m}\sum_{i=1}^m\rho^2(\xrv_{(i,n)}(\xrv),\xrv)\).
\end{lemma}
\begin{proof}
    To derive our result, we start by stating Lemma 1 from \cite{Kohler2006} without proof:
    \begin{lemma}[Lemma 1 \cite{Kohler2006}]
    \label{lem:kohler}
        Assume there exists a constant \(\beta>2r\) such that \(\Eold\norm{\xrv}^\beta<\infty\) and \(d>2r\) for constants  \(0<r\leq 1\), \(c_{11}>0\). Then
        \eq{
            \int \E\min\{\norm{\xrv_{(1,n)}(\xv)-\xv}^{2r},1\}\mu(d\xv) &\leq c_{11}n^{-\frac{2r}{d}}.
        }
    \end{lemma}
    
    Notice that in our case \(\rho^2(\xv,\xv')\leq \sigma_f^2\) for all \(\xv,\xv' \in \mc{X}\). Hence,
    \eq{
    \E{\rho^2(\xrv_{(1)}(\xrv),\xrv)} &= \sigma_f^2\E{\min\left\{\frac{\rho^2(\xrv_{(1)}(\xrv),\xrv)}{\sigma_f^2},1\right\}}
    }
    and now applying \ref{ass:kernel_dist} and \cref{lem:kohler} with \(p=2r\)
    \eq{
        &\leq \sigma_f^2 \int \E{\min\left\{\frac{L}{\sigma_f^2l^p}\norm{\xrv_{(1)}(\xv)-\xv}^p,1\right\}}\mu(d\xv) \\
        &= \sigma_f^2\int \E{\min\left\{\norm{\zrv_{(1)}(\zv)-\zv}^p,1\right\}}\nu(d\zv) \\
        &\leq \sigma_f^2\, c_{11} n^{-\frac{p}{d}}
    }
    for \(\zv=\frac{L^{\frac{1}{p}}\xv}{\sigma_f^{2/p}l}\) and using that \(\E\norm{\zrv}^\beta<\infty \iff \E\norm{\xrv}^\beta<\infty\). 

    Now, to obtain the bound on the average nearest-neighbour distance we adopt a technique from page 95 of \cite{Gyorfi2010}: We divide the full dataset into \(m+1\) non-overlapping segments such that the first \(m\) have length \(\floor*{\frac{n}{m}}\) and \(\tilde{\xrv}_j^{\xv}\) is the nearest-neighbour to \(\xv\) from the \(j^{th}\) segment:
    \eq{
        \Ex{\avg{\eps_\bullet}} &= \avg{\Ex{\rho^2(\xrv_{(\bullet)}(\xrv),\xrv)}} \\
        &\leq \avg{\Ex{\rho^2(\tilde{\xrv}_\bullet^{\xrv},\xrv)}} = \Ex{\rho^2(\tilde{\xrv}^{\xrv}_1,\xrv)} \\
        &= \Ex{\rho^2(\xrv_{(1,\floor*{\frac{n}{m}})}(\xrv),\xrv)} \\
        &\leq \sigma_f^2c_{11}\floor*{\frac{n}{m}}^{-p/d} \leq C_{\bm{\theta},d}\left(\frac{n}{m}\right)^{-p/d}.
    }
\end{proof}


We use the results just derived to establish upper bounds on the convergence rate of the expected value of the MSE over the input space. This is then straightforwardly adapted to provide both a bound on the corresponding variance and hence a probablistic guarantee of performance, via application of Cantelli's inequality.

\begin{theorem}[MSE convergence rate upper bound (well-specified)]
\label{thm:mse_conv}
    Under \hyperref[ass:convergence]{Assumptions (A1-4)}
    \begin{align*}
        \Ex{f_n^{\MSE}(\bm{\theta})} \lesssim \,&\sigma_\xi^2(1 + m^{-1}) + 2C_{\bm{\theta},d}\left(\frac{n}{m}\right)^{-\frac{p}{d}}.
    \end{align*}
    where \(C_{\bm{\theta},d}\) is a constant depending only on \(\bm{\theta}\) and \(d\).
\end{theorem}

\begin{proof}
    To prove this we use a series of intermediate results already given. 
    Taking \cref{lem:MSE_UB} we see that \(\Ex{f_n^{\MSE}(\bm{\theta})}\leq \sigma_\xi^2 + \sigma_f^2 - \sigma_f^2(1+\sigma_\xi^2/m\sigma_f^2)^{-1} + 2\Ex{\avg{\eps_\bullet}}(1+\sigma_\xi^2/m\sigma_f^2)^{-1} + \Ex{\bigO{\eps^2}}.\) Expanding the denominator (\((1+\sigma_\xi^2/m\sigma_f^2)^{-1}\)), we obtain an upper bound of \(\sigma_\xi^2(1+m^{-1}) + 2\Ex{\avg{\eps_\bullet}}\), since the terms in \(m^{-2}\) are negative and neglecting the term \(\Ex{\bigO{\eps^2}}=o(\Ex{\eps})\). 
    

\end{proof}

\begin{lemma}[Minimum MSE]
\label{lem:min_mse}
    Given a mean-zero WSRF with i.i.d additive noise, the minimum MSE achievable with an \(m\)-point unbiased linear predictor is \(\sigma_\xi^2(1+m^{-1})\).
\end{lemma}
\begin{proof}
    Define a linear predictor by the expression \(\bm{w}^T\yv\). At a test point \(\xv^*\), \(f^*=f(\xv^*)\) must satisfy \(\E{f^*-\bm{w}^T\yv}\) if the prediction is unbiased. Suppose we can select observations such that \(\yv=f^*\1 + \bm{\delta}+\bm{\xi}\) with \(\bm{\xi}\sim\mc{N}(0,\sigma_\xi^2)\), then \(0=\E{f^*-\bm{w}^T\yv}=\E{f^*}(1-\bm{w}^T\1) -\bm{w}^T\bm{\delta}\implies\bm{w}^T\delta=0.\) if \(\E{f^*}=0.\) Thus we have an unbiased linear predictor of repeated noisy observations at the test location. Minimising the MSE w.r.t \(\bm{w}\) then gives \(\bm{w}_0=\frac{1}{m}\1\) and the result follows.
\end{proof}

In addition, we can find a related result for the variance:
\begin{lemma}[MSE convergence rate variance upper bound (well-specified)]
\label{lem:mse_conv_var}
   \begin{align*}
        \Var{f_n^{\MSE}(\bm{\theta})} \leq\, & \tilde{C}^m_{\bm{\theta},d}\left(\frac{n}{m}\right)^{-\frac{p}{d}}+ R_{mn},
    \end{align*}
    with \(\tilde{C}^m_{\bm{\theta},d}=4\sigma_\xi^2(1+m^{-1}) C_{\bm{\theta},d}\) and \(R_{mn}=\bigO{n^{-\frac{2p}{d}}}\).
\end{lemma}

\begin{proof}
    The proof is given in \cref{proof:mse_conv_var} but follows from the result of 
    \cref{thm:mse_conv} and that the best possible MSE is given by \(\sigma_\xi^2(1+m^{-1})\) (\cref{lem:min_mse}).
\end{proof}
Henceforth the term \(R_{mn}\) is taken to be the remainder term of the same order as defined here.

\begin{lemma}[Probabilistic MSE guarantees]
\label{lem:prob_mse_bound}
    For constants \(\varepsilon,\delta > 0\), 
    if \( n \gtrsim m\left[\frac{3\varepsilon}{2C_{\bm{\theta},d}}+\frac{\sigma_\xi^2(1+m^{-1})}{C_{\bm{\theta},d}\delta}\right]^{-\frac{d}{p}}\)
    then with probability at least \(1-\delta\),
    \[f_n^{\MSE}(\bm{\theta}) < \sigma_\xi^2(1+m^{-1}) + \varepsilon.\]
\end{lemma}
\begin{proof}
    We use a modified form of Cantelli's inequality to account for the fact that we only have results for upper and lower bounds on the mean and variance. To obtain the stated result we substitute these bounds, bound the RHS from above by a constant \(\delta>0\) and solve the resulting inequality for \(n\). We defer the details of this to \ref{proof:prob_mse_bound}.
\end{proof}


\subsection{Under Misspecification}
\label{sec:mis_rate}
We now address the misspecified case and again begin with a lemma bounding the MSE in terms of the kernel induced distance functions, only now we must take into account both the generative and modelling kernels and the various potential forms that ``misspecification'' can take. For example, a minimum case would be when only the kernel hyperparameters have incorrect values. For more specificity, we divide the hyperparameter vector into ``external'' (\(\bm{\phi}\)) and ``internal'' (\(\bm{\lambda}\)) parameters. We set \(\bm{\phi}=(\sigma_f^2,\sigma_\xi^2)\) but leave the internal parameters to refer to whichever additional parameters the kernel function (\(c(\cdot,\cdot)\)) may utilise, such as, but not limited to, a single lengthscale. We define the following set of misspecifications (not including the assumptions throughout of non-Gaussian noise and WSRF \ref{ass:wsrf}):

\begin{enumerate}[label=(\alph*)]
    \item Correct kernel function (\(c(\cdot,\cdot)\)), with incorrect parameters (\(\thv\)).
    \begin{enumerate}[label=(\roman*)]
        \item Incorrect \emph{external} parameters, \(\bm{\hat{\phi}}\).
        \item Incorrect \emph{internal} parameters, \(\bm{\hat{\lambda}}\).
        \item All incorrect \(\thv\).
    \end{enumerate}
    \item Incorrect kernel function (\(\hat{c}(\cdot,\cdot)\)). 
        \begin{enumerate}[label=(\roman*)]
            \item \(c=c_{RBF},\hat{c}=c_{Exp}\)
            \item \(c=c_{Exp},\hat{c}=c_{RBF}\)
        \end{enumerate}
    \label{misspecs}
\end{enumerate}
Note that when comparing (b)(i) and (b)(ii) we could consider the misspecification to be given in terms of the additional Mat\'{e}rn smoothness parameter \(\nu\).

\begin{lemma}[Upper bound on the misspecifed model MSE]
\label{lem:misspec_MSE_UB}
    Under all modelling \hyperref[ass:convergence]{Assumptions (A1-4)} and generic model misspecification ((a)(iii) or (b)),
    \begin{align*}
    f_n^{\MSE}(\thv) \leq& \sigma_\xi^2(1+m^{-1}) + 2\frac{\sigma_f^2}{\hat{\sigma}_f^2}(\avg{\hat{\eps}_\bullet}-\hat{\eps}^{\bullet}_{\min}) + 2\avg{\eps_\bullet} \pm \bigO{\eps^2+m^{-2}}.
    \end{align*}
\end{lemma}
\begin{proof}
    We obtain this result via Neumann expansion of the kernel inverse in a perturbed form, where the perturbation is given by \(\eps_{ij}\) terms. In particular, we use the results given in \cref{lem:Kinv} and \cref{lem:1TKinv1}. A full proof is in \ref{proof:misspec_MSE_UB}.
\end{proof}

We can now state the convergence rate result for a misspecified model in terms of both the generative and predictive model parameters by combining \cref{lem:misspec_MSE_UB} with the intermediate results derived in \cref{sec:rates}.
\begin{theorem}[Misspecified MSE convergence rate]
    Under the same conditions as \cref{thm:mse_conv} but with misspecification type (a)(iii) or (b),
    \label{thm:mse_conv_mis}
        \begin{align*}
        \Ex{f_n^{\MSE}(\thv)} \leq &\sigma_\xi^2(1 + m^{-1}) + 2C_{\bm{
        \theta},d}\left(\frac{n}{m}\right)^{-\frac{p}{d}} + 2\frac{\sigma_f^2}{\hat{\sigma}_f^2}C_{\thv,d}\left(\frac{n}{m}\right)^{-\frac{\hat{p}}{d}} + R_{mn}.
    \end{align*}
    \(R_{mn}\) matches that given in \cref{lem:mse_conv_var} asymptotically.
\end{theorem}

\begin{proof}
    The result is gained by combining \cref{lem:misspec_MSE_UB} with \cref{lem:eps_i} (see \cref{proof:mse_conv_mis}).
\end{proof}

\begin{lemma}[MSE convergence rate variance upper bound (misspecified)]
\label{lem:mse_conv_var_mis}
   \begin{align*}
        \Var{f_n^{\MSE}(\thv)} \leq & 2\tilde{C}^m_{\bm{\theta},d}\left(\frac{n}{m}\right)^{-\frac{p}{d}}+ 2\tilde{C}^m_{\thv,d}\left(\frac{n}{m}\right)^{-\frac{\hat{p}}{d}} + R_{mn},
    \end{align*}
    with \(\tilde{C}^m_{\thv,d}=2\sigma_\xi^2(1+m^{-1}) \frac{\sigma_f^2}{\hat{\sigma}_f^2} C_{\thv,d}\) and \(\tilde{C}^m_{\bm{\theta},d}\) as before.
\end{lemma}
\begin{proof}
    The proof construction mirrors that of \cref{lem:mse_conv_var} using \cref{thm:mse_conv_mis} rather than \cref{thm:mse_conv} (details are given in \cref{proof:mse_conv_var_mis}).
\end{proof}

As in the well-specified case, we derive a probabilistic guarantee, but we leave the details to \cref{lem:prob_mse_bound_mis} due to the added complexity of both the derivation and result.


\subsubsection{Calibration}
\label{sec:calibration}
When well-specified, there is no sense in considering the convergence rate of the weak-calibration measure \(f_n^{\CAL}\) (defined in \cref{sec:notation}) since it is guaranteed to be fixed at 1 in this case. However, in the misspecified case, we can attempt to determine the rate at which the metric is expected to be in some interval around the ratio \(\frac{\sigma_\xi^2}{\hat{\sigma}_\xi^2}\). To estimate this quantity in addition to both upper and lower bounds on \(f_n^{\MSE}\) we require bounds on \(\predvar[\hat]_N\).
\begin{lemma}[Calibration convergence]
\label{lem:cal_conv}
    Under the same conditions as \cref{lem:misspec_MSE_UB} with \(\hat{\eps}^\bullet_{\cdot}=\hat{\eps}^\bullet_{\cdot}(n)\), \(m=m(n)\),
    \eq{
        f_n^{\CAL}(\thv) &\geq \frac{\sigma_\xi^2}{\hat{\sigma}_\xi^2}\left[1 + \frac{2\avg{\hat{\eps}_{\bullet}}}{\hat{\sigma}_\xi^2(1+m^{-1})}\right]^{-1}, \\ 
        f_n^{\CAL}(\thv) &\lesssim \left(\frac{\sigma_\xi^2}{\hat{\sigma}_\xi^2} + \frac{2}{\hat{\sigma}_\xi^2}\left(\frac{\sigma_f^2}{\hat{\sigma}_f^2}(\avg{\hat{\eps}_{\bullet}}-\hat{\eps}^{\bullet}_{\min})+\avg{\eps_{\bullet}}\right)\right) \left[1 + \frac{2\hat{\eps}^{\bullet}_{\min} + \avg{\hat{\eps}_{\bullet\bullet}}}{\hat{\sigma}_\xi^2(1+m^{-1})}\right]^{-1}.
    }
\end{lemma}
\begin{proof}
    Note that \(f_n^{\CAL}(\thv)=\frac{f_n^{\MSE}(\thv)}{\predvar[\hat]_N}\). To derive a lower bound observe that the lower bound for the numerator is \(\sigma_\xi^2(1+m^{-1})\) (\cref{lem:min_mse}). We then require an upper bound for the denominator. We use \cref{lem:misspec_MSE_UB} and derive a lower bound for the predictive variance to obtain the upper bound.
    \Cref{proof:cal_conv} contains the details.
\end{proof}

We can convert the above result to one depending on \(n\) as we did in the MSE case:
\begin{theorem}[Calibration convergence]
\label{thm:cal_conv}
    Under the same conditions as \cref{thm:mse_conv_mis},
    \eq{
        \Ex{f_n^{\CAL}(\thv)} &\geq \frac{\sigma_\xi^2}{\hat{\sigma}_\xi^2}\left[1 + \frac{2C_{\thv,d}}{\hat{\sigma}_\xi^2(1+m^{-1})}\left(\frac{n}{m}\right)^{\frac{-\hat{p}}{d}}\right]^{-1} \\
        \Ex{f_n^{\CAL}(\thv)} & \lesssim \frac{\sigma_\xi^2}{\hat{\sigma}_\xi^2} + \frac{2}{\hat{\sigma}_\xi^2}\bigg(\frac{\sigma_f^2}{\hat{\sigma}_f^2} C_{\thv,d}\left(\frac{n}{m}\right)^{-\frac{\hat{p}}{d}}  + C_{\bm{\theta},d}\left(\frac{n}{m}\right)^{-\frac{p}{d}}\bigg).
        }
\end{theorem}
\begin{proof}
    First we prove the lower bound by noting that, for \(x>0\), \((1+x)^{-1}\) is convex, so we can simply insert the expectation over \(\hat{\eps}\) (using \cref{lem:eps_i}) and apply Jensen's inequality to obtain the result.

    For the upper bound, we first see that the term in square brackets is bounded from below by 1 and then directly apply \cref{lem:eps_i} again.
\end{proof}

\begin{remark}[NLL convergence]
    Due to the close relationship between weak-calibration and NLL, we could find an upper bound for \(\Ex{f_n^{\NLL}}\) by utilising the upper bound for \(\predvar[\hat]_N\) given in the proof above and exploiting concavity of \(\log\) to apply Jensen's inequality.
\end{remark}

\section{Experiments}
\label{sec:exp}

We run a series of simulations as per Algorithm 1 of \cite{GPnn} to investigate the empirical convergence behaviour under various conditions, and levels of misspecification (\ref{sec:mis_rate}), to assess the validity of the theoretical results given as well as provide contextual behaviour in cases where the theory provides less interpretable results. 

Data were generated by first sampling inputs according to \(\xrv\sim\mc{N}(0,d^{-1}I_d)\) to normalise the expected length of the vectors (i.e. so that \(\Eold\norm{x}^2=1\)). The benefit of using such a normalisation is that it is equivalent to scaling the lengthscale by a factor of \(\sqrt{d}\) and allows a reasonable comparison between lengthscales of datasets of different dimension, to ascertain the general behaviour of the models and clarify the impact of dimensionality beyond the (necessary) lengthscale rescaling.

Sample \(y\) values were generated using a predefined set of kernel parameters and GPnn predictions made subsequently on a hold-out (synthetic) dataset. Performance metrics (as defined previously) were measured on this hold-out set and the process repeated for multiple values of \(n\), input dimension \(d\) and values of \(\hat{l}\) over an interval around \(l\). We estimated the intercept and slope of the \(\log\)-\(\log\) plot of \(\MSE-\sigma_\xi^2(1+m^{-1})\) against \(n\), measured via standard python functions. Further details are provided in \ref{appx:experiments}.


Additionally, for sufficiently small lengthscales, the sizes of \(n\) used in the simulations are not large enough to ensure that in the misspecified case the Neumann expansion used in the proof of the upper bound converges (\cref{lem:Kinv}). As such, we cannot claim that we expect the upper bound to hold in these regions. It is perhaps not surprising that in these cases performance is poor and convergence is not observed, since all of the points are effectively far apart and heuristically we would expect much larger values of \(n\) to be required to reach near-limiting behaviour.

\subsection{Misspecification (a)(ii)}
\def\FigureScale{0.45}

\begin{figure}
\centering
    \subcaptionbox{\label{fig:rates_matched_rbf_rbf}}{\includegraphics[scale=\FigureScale]{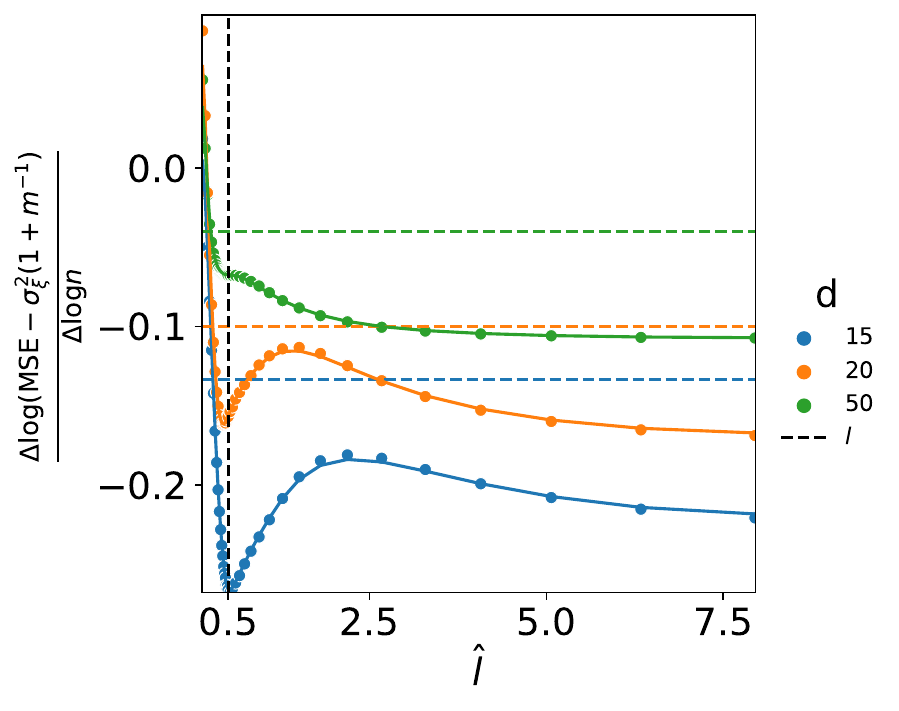}}
    \hfill
    \subcaptionbox{\label{fig:rates_matched_exp_exp}}{\includegraphics[scale=\FigureScale]{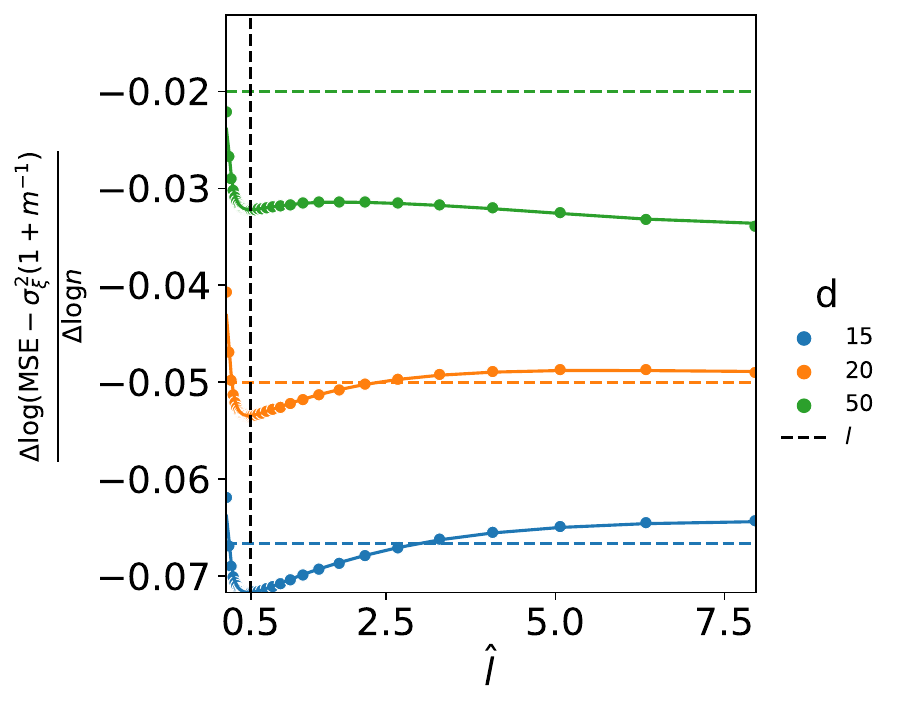}}
\caption{Empirical convergence rate as a function of lengthscale with matched external parameters: \subref{fig:rates_matched_rbf_rbf}  (\(c_{RBF}\)-\(\hat{c}_{RBF}\)) and \subref{fig:rates_matched_exp_exp} (\(c_{Exp}\)-\(\hat{c}_{Exp}\)). Dashed coloured horizontal lines represent the upper bound rate (given by \(-2/d\)) where each colour corresponds to the relevant value of \(d\).}
  \label{fig:rates_matched}
\end{figure}



\Cref{fig:rates_matched_rbf_rbf} and \Cref{fig:rates_matched_exp_exp} show the observed empirical gradient of the MSE with respect to \(n\) on results from simulations on synthetic data where only the kernel internal parameters are misspecified for the RBF and Exp kernels respectively. Dashed coloured horizontal lines represent the upper bound rate (given by \(-p/d\)) where each colour corresponds to the relevant value of \(d\).
For figures \ref{fig:rates_matched_rbf_rbf} and \ref{fig:rates_matched_exp_exp} we observe minima close to the true lengthscale \(l\) followed by an increase as \(\hat{l}\) increases, to a maximum in the case of \cref{fig:rates_matched_rbf_rbf}. 

\subsection{Misspecification (b)(i-ii)}
\begin{figure}
\centering
    \subcaptionbox{\label{fig:rates_matched_rbf_exp}}{\includegraphics[scale=\FigureScale]{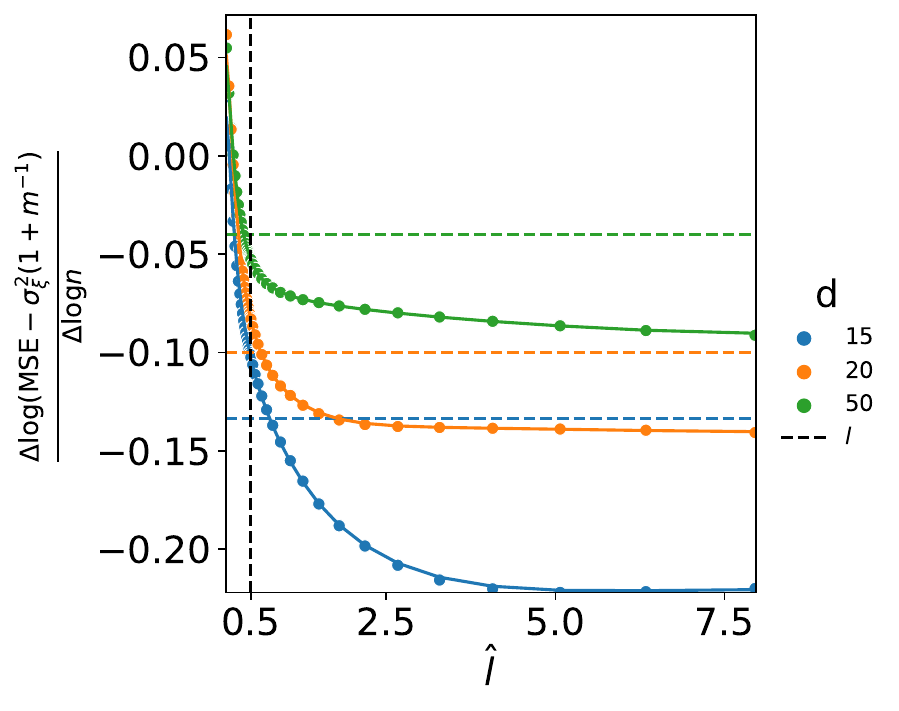}}
    \hfill
    \subcaptionbox{\label{fig:rates_matched_exp_rbf}}{\includegraphics[scale=\FigureScale]{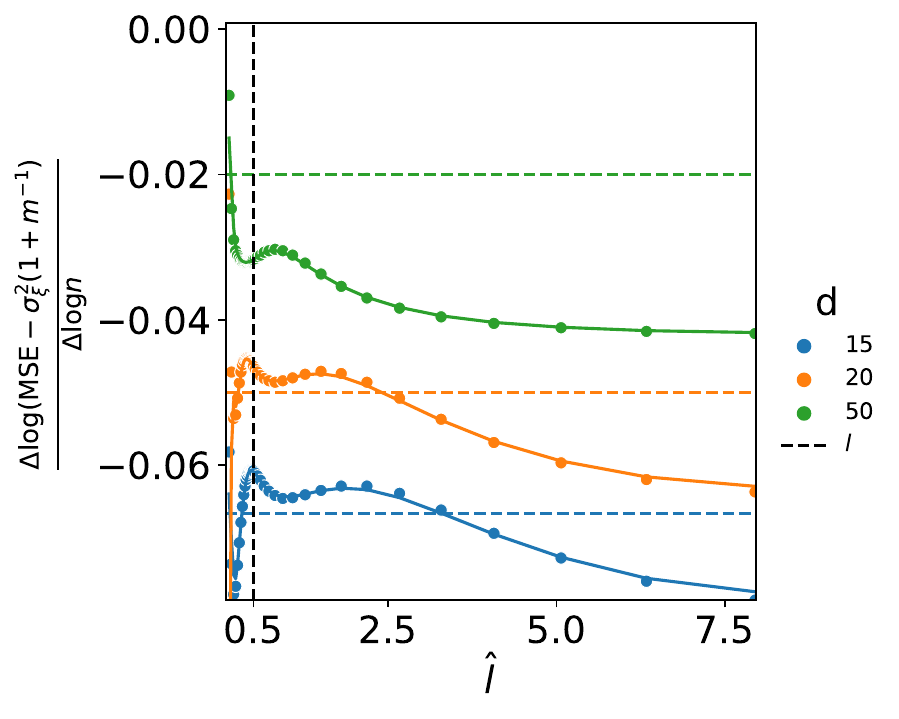}}
\caption{Empirical convergence rate as a function of lengthscale with matched external parameters: \subref{fig:rates_matched_rbf_rbf}  (\(c_{RBF}\)-\(\hat{c}_{Exp}\)) and \subref{fig:rates_matched_exp_exp} (\(c_{Exp}\)-\(\hat{c}_{RBF}\)). Dashed coloured horizontal lines represent the upper bound rate (given by \(-2/d\) and \(-1/d\) respectively) where each colour corresponds to the relevant value of \(d\).}
  \label{fig:rates_mismatched}
\end{figure}


In addition to the previous section we run simulations using mismatched kernels. Although the external parameters were matched in this case, this is no longer as meaningful. We see that the rates (at least for \(\hat{l}>l\)) in \Cref{fig:rates_matched_rbf_exp} are generally significantly lower than those in \Cref{fig:rates_matched_exp_rbf}.

Interestingly, in both cases we no longer observe well-defined minima, but \cref{fig:rates_matched_rbf_exp} shows a gradual decrease in rate as \(\hat{l}\) increases which perhaps reflects the relative smoothness of the RBF vs Exp kernel; i.e. a random function generated from an RBF kernel with lengthscale \(l\) will be very smooth whilst a much less-smooth Exp kernel may need a much larger lengthscale to model it effectively. Conversely, \cref{fig:rates_matched_exp_rbf} shows a steep downward trend for \(\hat{l}<l\) which matches our expectations that, in this case, a \emph{smaller} lengthscale for the smoother RBF model kernel better models the rougher Exp-generated data. In fact, using longer lengthscales tends to be a reliable way to obtain moderately effective performance in both (b)(i) and (b)(ii).

We can go some way to explain the relatively good performance of the model with long lengthscales with the following lemma:
\begin{lemma}[MSE as \(\hat{l}\rightarrow\infty\)]
\label{lem:mse_limitl}
    With only \hyperref[ass:convergence]{assumptions (A1-3)}, as \(\hat{l}\rightarrow\infty\) we have that
    \eq{
        \Ex{f_n^{\MSE}(\thv)} &\leq \sigma_\xi^2(1+m^{-1}) + 2C_{\bm{\theta},d}\left(\frac{n}{m}\right)^{-\frac{p}{d}} \pm \bigO{m^{-2}}.
    }
\end{lemma}
\begin{proof}
    As \(\hat{l}\rightarrow\infty\), \(\hat{K}\rightarrow \hat{K}^{\infty}\) and we can apply existing results to compute all of the quantities. After some algebra we obtain \(f_n^{\MSE}(\thv) = \sigma_\xi^2(1+m^{-1}) + 2\avg{\eps_\bullet} - \avg{\eps_{\bullet\bullet}} \pm \bigO{m^{-2}}\) to which we simply apply \cref{lem:eps_i}. A full proof is given in \ref{proof:mse_limitl}.
\end{proof}

This effectively describes the behaviour of conventional \(k\)-NN prediction (notice that the RHS has no dependence on assumed (hatted) parameters), which we might view as a natural benchmark for GPnn. The fact that this rate appears to be \emph{better} than that derived in the fully misspecified case (since the RHS of \cref{lem:mse_limitl} is less than or equal to the RHS of \cref{thm:mse_conv_mis}) may demonstrate that these bounds are not tight and fail to describe the full dependence on lengthscale which might more clearly reveal the superior empirical performance observed for GPnn at non-extreme values of \(\hat{l}\) (see Figures \ref{fig:conv_matched_rbf_rbf}, \ref{fig:conv_matched_rbf_exp}, \ref{fig:conv_matched_exp_rbf}, \ref{fig:conv_matched_exp_exp} and \cref{tab:best_kernels} in \ref{appx:experiments} for evidence). The above relationship between GPnn and conventional \(k\)-NN prediction raises the question as to what range of lengthscales (and other parameters) might we expect GPnn to outperform \(k\)-NN; \ref{appx:sims} addresses this point in a series of figures that demonstrate a generally sizeable interval for \(\hat{l}\) for which this level of accuracy is achieved. It is notable that \(k\)-NN performs more competitively for data drawn from a GP with an exponential kernel, as we expect. 

We also note that as \(d\) increases the curves become flatter, which may be a result of the decreasing size of fluctuations of the off-diagonal terms of the Gram matrix in this case, as explored in \ref{appx:highd}.

\section{Discussion}
\label{sec:discussion}
The convergence rates derived in the preceding sections assume a fixed value for \(m\), the number of nearest neighbours. This is to allow a straightforward interpretation of the results of the model as introduced in \cite{GPnn} (where \(m\) was set to a somewhat arbitrary 400). As seen in the limiting results this leaves a (small) bias that is removed if \(m=m(n)\xrightarrow{n\rightarrow\infty}\infty\). When such an \(m=\mc{O}\big(n^\frac{p}{p+d}\big)\) is chosen, we obtain rates more directly comparable with those seen in the literature: \(\mc{O}\big(n^{-\frac{p}{p+d}}\big)\). This rate can be shown to be minimax optimal under certain circumstances and under conditions more restrictive than those we consider here. In the case of \(\nu\leq1\), if we choose \(m=n^{\frac{2\nu}{2\nu+d}}\) then we obtain the optimal rate \(\mc{O}\big(n^{-\frac{2\nu}{2\nu+d}}\big)\) for \(\nu\)-regular functions over \([0,1]^d\). See \cite{Tsybakov2008} for a discussion of optimal rates of convergence under standard nonparametric assumptions and e.g. \cite{van_der_vaart_information_2011} for their usage in the context of GPs. As in the example just given, the literature tends to assume the covariates are restricted to a bounded input space commonly taken to be \([0,1]^d\). In contrast, our results hold with a moment condition on the covariate distribution over a (possibly unbounded) space. We are unaware of results in the literature pertaining to the optimal rates under these precise conditions. 

In \cref{sec:exp} we make headway in explaining how lengthscale and dimension impact the convergence rate and how this relates to the performance of straight \(k\)-NN performance. Despite a face-value equivalence in asymptotic rates we expect GPnn to outperform \(k\)-NN in all but the most extreme cases. The resilience of GPnn in the face of the curse-of-dimensionality is also striking when contrasted to \(k\)-NN, as demonstrated on the CTslice dataset in \cref{tab:best_kernels}.

\section{Further Work}
\label{sec:further}
Although we have made progress in understanding the behaviour of GPnn convergence as a function of generative and predictive model parameters, the results are given as upper bounds that we do not believe to be, in general, tight. As such, a practitioner wishing to evaluate sufficient lower bounds on \(n\) for which they can give guarantees on prediction MSE will be restricted to unnecessarily large values. This might be improved by a more thorough understanding of the size of the constant factor.

Additionally, the authors are actively engaged in extending this analysis in a comparable way to alternative GP approximation techniques within the framework presented here and in \cite{GPnn} where training and prediction have been separated.

%% file: appendix.tex
\appendix

\input{appendices/preliminary}

\input{appendices/convergence_rates}

\input{appendices/misspec_rates}
\input{appendices/more_results}
\input{appendices/experiments}
\input{appendices/nngp_approx}

%% file: appendices/preliminary.tex
\section{Preliminary results}
\label{sec:prelim}
\subsection{Kernels satisfying \texorpdfstring{\cref{ass:kernel_dist}}{}}
\label{appx:assumptions}
First, a technical lemma to allow us to bound the sum of the inverse kernel that we make extensive use of in what follows.
 \begin{lemma}
    \label{lem:1TA1}
    For \(\alpha,\beta \in \mathbb{R}_+\), \(A \in \mathbb{R}^{n\times n}\), let \(\alpha I \preceq A \preceq \beta I\), then \(\beta^{-1} I \preceq A^{-1} \preceq \alpha^{-1}I\) and \(\lambda_1(A)^{-1}\leq n^{-1}\1^TA^{-1}\1 \leq \lambda_n(A)^{-1}\) where \(\lambda_1(A) = \max_{i\in[n]}\lambda_i(A)\) and \(\lambda_n = \min_{i\in[n]}\lambda_i(A)\).
\end{lemma}

\begin{proof}
    The first part follows from any standard text e.g. Corollary 7.7.4 (page 495) of  \cite{Horn1985} and the second by simply choosing \(\beta=\lambda_1(A)\) and \(\alpha = \lambda_n(A)\).
\end{proof}

Using \cite{NIST:DLMF} we see that \(1+x < \e^x\). For sufficiently small \(\Delta/l\) we can find lower bounds using \(\e^{-x} < 1-\half x\) for \(x\in (0,1.5936...]\). Hence, for common kernels with exponential terms, such as those that follow, we can obtain Euclidean upper bounds on their respective metric functions.
\eq{
    k_{SE}(\Delta;l) &= \sigma_f^2\e^{-\frac{\Delta^2}{2l^2}}, \\
    k_{Exp}(\Delta;l) &= \sigma_f^2\e^{-\frac{\Delta}{l}}, \\
    k_{M,\nu}(\Delta;l) &= \sigma_f^2\frac{2^{1-\nu}}{\Gamma(\nu)}\left(\sqrt{2\nu}\frac{\Delta}{l}\right)^\nu \mc{K}_\nu\left(\sqrt{2\nu}\frac{\Delta}{l}\right)
}
where \(\mc{K}_\nu\) is the modified Bessel function of the second kind.

Using the exponential bound(s) given directly above, we obtain the following bounds on the associated distance functions:
\eq{
    \frac{\Delta^2}{4l^2}< \rho^2_{SE}(\Delta;l)/\sigma_f^2 & < \frac{\Delta^2}{2l^2} \\
    \frac{\Delta}{2l} < \rho^2_{Exp}(\Delta;l)/\sigma_f^2 & < \frac{\Delta}{l} 
}

We can combine the above with a mixture representation of the general Mat\'{e}rn kernel (see e.g. \cite{Tronarp2018a}) to find a bound for the distance (for \(\nu>1\)):
\eq{
    k_{M,\nu}(\Delta;l) &= \int_0^\infty k_{SE}(\Delta;\sqrt{s})G(s;\nu,\nu/l^2)ds \\
    &\geq \int_0^\infty (1-\frac{\Delta^2}{2s})G(s;\nu,\nu/l^2)ds \\
    &= 1 - \half\Delta^2\E_G{s^{-1}} \\
    &= 1 -  \frac{\Delta^2}{2l^2}\frac{\nu}{\nu-1}
}
and hence
\eq{
    \rho^2_{M,\{\nu>1\}}(\Delta;l)/\sigma_f^2 &< \frac{\Delta^2}{2l^2}\frac{\nu}{\nu-1}.
}
Similarly, for the Rational Quadratic kernel, we obtain:
\eq{
    \rho^2_{RQ}(\Delta;l)/\sigma_f^2 &< \frac{\Delta^2}{2l^2}
}
with a directly comparable argument to that given for the Mat\'{e}rn, using the spectral mixture distribution given in \cite{Tronarp2018a}.

The upper bound is what we need to make the later argument and we note that none of the above (upper) bounds require constraints on \(\Delta\) or \(l\). However, the following argument, which uses an alternative approach to find Euclidean bounds for the Mat\'{e}rn kernel with \(\nu<1\), \emph{does} have such a restriction:
\begin{lemma}[Mat\'{e}rn kernel bound \(\nu<1\)]
    For \(\nu<1\) and \(\Delta/l < \left(\frac{2^\nu\Gamma(\nu)}{\nu^\nu(1-\nu)\abs{\Gamma(-\nu)}}\right)^{\frac{1}{2\nu}}\) we have that
    \eq{
        \rho^2_{M,\{\nu<1\}}(\Delta;l)/\sigma_f^2 &\leq \frac{\abs{\Gamma(-\nu)}}{\Gamma(\nu)}\nu^\nu\left(\frac{\Delta}{l}\right)^{2\nu}.
    }
\end{lemma}
\begin{proof}
Using equation 3.2 from \cite{Geoga2021}, we assume that for sufficiently small \(\Delta\) we can truncate the series and we retain terms \(k\leq 2\):
\eq{
    K_\nu(x) &= \half\Gamma(\nu)\left(\frac{x}{2}\right)^{-\nu}+\half\Gamma(-\nu)\left(\frac{x}{2}\right)^\nu + \left(\frac{x}{2}\right)^2\half \\
    &\quad\cdot\left[\Gamma(\nu)\left(\frac{x}{2}\right)^{-\nu}\frac{\Gamma(1-\nu)}{\Gamma(2-\nu)}+\Gamma(-\nu)\left(\frac{x}{2}\right)^\nu\frac{\Gamma(1+\nu)}{\Gamma(2+\nu)}\right] \\
    &= \half\Gamma(\nu)\left(\frac{x}{2}\right)^{-\nu} - \half\abs{\Gamma(-\nu)}\left(\frac{x}{2}\right)^\nu + \half\left(\frac{x}{2}\right)^{2-\nu} \frac{\Gamma(\nu)}{1-\nu} - \half\frac{\abs{\Gamma(-\nu)}}{1+\nu}\left(\frac{x}{2}\right)^{2+\nu}
}
where we have written \(-\abs{\Gamma(-\nu)}\) to indicate that \(\Gamma(-\nu)<0\) for \(\nu \in (0,1)\).
We now substitute this into the expression for the general mat\'{e}rn kernel (defining \(\tilde{\Delta}=\sqrt{2\nu}\Delta/l\)):
\eq{
    c_{M,\{\nu<1\}}(\Delta;l) &= \frac{2^{1-\nu}}{\Gamma(\nu)}\tilde{\Delta}^\nu\bigg\{\half\Gamma(\nu)2^\nu\tilde{\Delta}^{-\nu} - \half\abs{\Gamma(-\nu)}2^{-\nu}\tilde{\Delta}^\nu + \half\frac{\Gamma(\nu)}{1-\nu}2^{\nu-2}\\
    &\quad\cdot\tilde{\Delta}^{2-\nu} - \half\frac{\abs{\Gamma(-\nu)}}{1+\nu}2^{-2-\nu}\tilde{\Delta}^{2+\nu}\bigg\}\\
    &\geq 1 - \frac{\abs{\Gamma(-\nu)}}{\Gamma(\nu)}\nu^\nu\left(\frac{\Delta}{l}\right)^{2\nu}
}
where the inequality follows if the following sufficient condition is satisfied (such that the second order terms contribute a positive term):
\eq{
    \Delta/l &< \left(\frac{2^\nu\Gamma(\nu)}{\nu^\nu(1-\nu)\abs{\Gamma(-\nu)}}\right)^{\frac{1}{2\nu}}.
}
\end{proof}
The above constraint is satisfied by, e.g. \(\Delta < 2.2l\) for \(\nu=0.4\) or \(\Delta < 3.05l\) for \(\nu=0.2\). Note also that for \(\nu=\half\) this gives a looser bound than the Exp bound given previously.

Note that although in the above arguments we have assumed an isotropic kernel with a single lengthscale, the above bounds hold in the case of independent lengthscales if we simply substitute \(l\) for \(l_{\min}=\min\{l_1,l_2,\dots,l_d\}\).

\subsubsection{Kernel bounds using Taylor expansion}
If we choose isotropic kernels for which the exponential bound given above is not sufficient, we may use a Taylor expansion to derive an appropriate Euclidean bound for which we can apply the nearest-neighbour convergence rate arguments.

As an example, take the periodic kernel:
\eq{
    k_{Per}(\Delta;l) &= \sigma_f^2\e^{-\frac{2\sin^2(\pi\Delta/r)}{l^2}} \\
    & > 1 - \frac{2}{l^2}\sin^2(\pi\Delta/r) \\
    \rho^2_{Per}(\Delta;l)/\sigma_f^2 &< \frac{2\pi^2}{r^2l^2}\Delta^2
}
where the last line follows from an expansion of the \(\sin\) function for which the bound holds provided \(\Delta < \frac{r\sqrt{6}}{\pi}\). We choose not to pursue this avenue further here due to the restriction that the bounds do not hold everywhere.

\subsection{Function Spaces}
\label{sec:function_spaces}
In this section we will elaborate slightly on the various function spaces invoked as part of this work. Common assumptions from the nonparametric literature, on e.g. \(k\)-NN methods, tend to be that the regression function obeys Lipschitz or H\"{o}lder continuity conditions. The GP community often makes use of both H\"{o}lder continuity as well as Sobolev spaces. For example, both \cite{van_der_vaart_information_2011} and \cite{Nieman2022} assume that an ``\(\alpha\)-regular'' data-generating function belongs to the space \(C^\alpha([0,1]^d\cap W^\alpha([0,1]^d)\) where \(C^\alpha([0,1]^d)\) represents the H\"{o}lder space of functions on \([0,1]^d\) with partial derivatives of order upto \(\floor{\alpha}\) where the highest order partial derivative satisfies a Lipschitz condition of order \(\eta\) where \(\alpha=\floor{\alpha}+\eta\) and \(\eta\in(0,1]\). On the other hand, \(W^\alpha([0,1]^d)\) contains functions on \([0,1]^d\) with partial derivatives upto order \(\floor{\alpha}\) (that are square integrable).
For a more in-depth discussion on the properties of Sobolev spaces, we refer the reader to e.g. \cite{sobolev_spaces_2012}. 

We can begin to reconcile these differences by noting that the \(k\)-NN literature assumes e.g. \(f_0\in C^r(\mc{X})\); so for \(r=\alpha\) and \(\mc{X}=[0,1]^d\), a generative function \(f_0 \in C^\alpha([0,1]^d\cap W^\alpha([0,1]^d)\) would satisfy both sets of conditions. 

The next step is to relate these conditions to those assumed in this manuscript. We take a random function drawn from a WSRF. In the more constrained case that the WSRF is in fact a GP, Kolmogorov's continuity theorem, in conjunction with \ref{ass:kernel_dist}, implies that such a function satisfies a H\"{o}lder condition \emph{locally}. (We are unaware of a related result in the more general WSRF case). In addition, if the function is known to satisfy a H\"{o}lder condition, then it can be shown that a corresponding version of \ref{ass:kernel_dist} holds: \(\abs{f(\xv)-f(\xv')}^2\leq C^2\norm{\xv-\xv'}^{2\gamma} \implies \E{\abs{f(\xv)-f(\xv')}^2} = 2\rho^2(\xv,\xv') \leq C^2\norm{\xv-\xv'}^{2\gamma}.\) In other words, a \(\gamma\)-regular function in the H\"{o}lder sense implies that \ref{ass:kernel_dist} holds with \(p=2\gamma\).

%% file: appendices/convergence_rates.tex
\section{Convergence rates in the well-specified case}
\label{appx:rates}
To obtain our convergence rates we require some results on the spectrum of the kernel matrix and its inverse as well as the properties of monotone matrices (which we will define in the section immediately following). 

\subsection{Matrix Monotonicity}
\label{sec:matrix_monotonicity}
\begin{definition}[Monotone Matrix]
\label{def:monotone}
    A real square matrix \(A\) is monotone if \(Av\geq 0 \iff v\geq 0\) (where \(\geq\) here is elementwise). 
\end{definition}

\begin{lemma}[Monotone Matrix]
\label{lem:monotone}
    Let \(A\) be a real square matrix. \(A\) is \emph{monotone} iff \(A^{-1}\geq 0.\)
\end{lemma}
\begin{proof}
    Suppose \(A\) is monotone. Let \(\xv\) be the \(i^{th}\) column of \(A^{-1}\). Then \(A\xv=\bm{e}_i\) with \(\bm{e}_i\) the base vector with \(i^{th}\) element 1. By monotonicity, \(\xv\geq \0\). On the other hand, suppose we have \(A^{-1}\geq 0.\) Then if \(A\xv\geq \0\), \(\xv=A^{-1}A\xv\geq A^{-1}\0\geq \0\), i.e. \(A\) is monotone.
\end{proof}

\begin{lemma}
\label{lem:mix_quad_prod_monotone}
    Let \(A\geq 0\) . Then for \(\xv,\yv \geq 0\), \(\xv^T A^{-1}\yv \geq 0\).
\end{lemma}
\begin{proof}
    If \(A\geq 0\) then \(A^{-1}\) is monotone. By the definition of monotonicity, \(\yv\geq 0 \implies A^{-1}\yv \geq 0\). Then \(\xv^T A^{-1}\yv = \xv^T\bm{w} \geq 0\) since both \(\bm{w}=A^{-1}\yv\) and \(\xv\) are elementwise nonnegative.
\end{proof}

\subsection{Spectral Properties of Gram Matrices}
\label{sec:spectrum}
A key result we make use of is \emph{Gershgorin's theorem} which we state without proof:

\begin{theorem}[Gershgorin]
\label{thm:gershgorin}
Given a matrix \(A \in \mathbb{C}^{n\times n}\), \(R_i=\sum_{j\neq i}\abs{a_{ij}}\) and define the disc \(D(a_{ii},R_i)\) of radius \(R_i\) centred on \(a_{ii}\), then
    every eigenvalue of \(A\) lies within at least one such disc.
\end{theorem}

This allows us, when combined with \cref{lem:1TA1}, to obtain upper and lower bounds on the sum of the elements of the kernel inverse, which is useful for computing our main results.

\begin{corollary}[Gram matrix inverse sum]
\label{cor:1TK1}
For a kernel Gram matrix \(K\) of size \(n\times n\) with elements \(K_{ij}=\sigma_f^2 - \eps_{ij} + \sigma_\xi^2\delta_{ij}\) and \(\eps^{\bullet\bullet}_{\min}=\min_{i,j}\eps_{ij}\)
    \[(\sigma_f^2 + \sigma_\xi^2 + (n-1)(\sigma_f^2-\eps^{\bullet\bullet}_{\min}))^{-1} \leq n^{-1}\1^TK^{-1}\1 \leq \frac{1}{\sigma_\xi^2}.\]
\end{corollary}
\begin{proof}
    The RHS is trivial and follows directly from \ref{lem:1TA1}. The LHS follows easily from Gershgorin's theorem (\ref{thm:gershgorin}) and \ref{lem:1TA1}: for our case \(a_{ii}=\sigma_f^2+\sigma_\xi^2\) and the radius is bounded from above by \(R_i = \sum_{j\neq i}\abs{\sigma_f^2 - \eps_{ij}} \leq (n-1)(\sigma_f^2-\eps^{\bullet\bullet}_{\min})\). Thus, \(\lambda_1(K_\xi) \leq \sigma_f^2 + \sigma_\xi^2 + (n-1)(\sigma_f^2-\eps^{\bullet\bullet}_{\min})\) and the result follows.
\end{proof}

We have that the predictive MSE (expectation over \((\yrv,\mathrm{y}^*)\)) is given by \(f_n^{\MSE} = \predvar_N = \sigma_f^2 + \sigma_\xi^2 - \kappa_N\), with \(\kappa_N = \kstar_N^TK_N^{-1}\kstar_N\) (as per standard GPR: \eqref{eq:gp_var_pred}) under the assumption that \(\yrv\sim\mc{WSRF}\) (see \cite{GPnn} for further commentary on this and in particular Lemma 13 for a proof that these expectations hold).
Here and henceforth we define \(K^\infty=\sigma_\xi^2I + \sigma_f^2\1\1^T\) and \(\gamma=\frac{\sigma_f^2}{\sigma_\xi^2+m\sigma_f^2}\). Note that \(\1\) is an eigenvector of the symmetric operator \(I-\gamma\1\1^T\) with eigenvalue \(1-m\gamma\) and so similarly an eigenvector of \((K^\infty)^{-1}=\sigma_\xi^{-2}(I-\gamma\1\1^T)\) with eigenvalue \(\sigma_\xi^{-2}(1-m\gamma)\). Hatted variants of these identities follow intuitively. We make use of these facts in the following lemmas and proofs.

So to find an upper bound on the MSE we find a lower bound for \(\kappa_N\):

\begin{lemma}[A lower bound on \(\kappa_N\).]
\label{lem:kappa}
With \(\hat{c}(\cdot)=c(\cdot)\) and \(\thv=\tv\), i.e. a well-specified model,
\eq{
    \kappa_N &\geq \frac{\sigma_f^2-2\avg{\eps_\bullet}}{1 + \sigma_\xi^2/m\sigma_f^2} -  \bigO{\eps^2}.
} 
\end{lemma}
\begin{proof}
    We start with the following decomposition of \(\kappa\):
    \begin{align}
        \kappa &= \sigma_f^4 \1^TK^{-1}\1 - 2\sigma_f^2\bm{\eps}^TK^{-1}\1 + \bm{\eps}^TK^{-1}\bm{\eps} \nonumber\\
        &\geq \sigma_f^4 \1^TK^{-1}\1 - 2\sigma_f^2\bm{\eps}^TK^{-1}\1. \label{eq:kappa_bound1}
    \end{align}
    From \cref{lem:1TKinv1}, we know that 
    \begin{align}
        \sigma_f^4\1^TK^{-1}\1 &= \sigma_f^2\left(1+\frac{\avg{\eps_{\bullet\bullet}}-m^{-1}\sigma_\xi^2}{\sigma_f^2}\right) \pm\bigO{\eps^2+m^{-2}}. \label{eq:1TKinv1}
    \end{align}

    We can plug the above and \cref{lem:epsKinv1} back into \eqref{eq:kappa_bound1} to give
    \eq{
        \kappa_N &= \sigma_f^2 - \frac{\sigma_\xi^2}{m} +\avg{\eps_{\bullet\bullet}}-\frac{2\avg{\eps_\bullet} }{1 + \sigma_\xi^2/m\sigma_f^2} \pm \bigO{\eps^2 + m^{-2}}.
    }
    \end{proof}

The preceding lemma makes use of the following supplementary lemmas:

\begin{lemma}[Approximating the kernel inverse; independent of misspecification assumptions]
\label{lem:Kinv}
Let \(E=K-K^\infty\). Then,
\begin{align*}
    K^{-1} &= 3(K^\infty)^{-1}  - 3(K^\infty)^{-1} K(K^\infty)^{-1}  + (K^\infty)^{-1} K(K^\infty)^{-1} K(K^\infty)^{-1} - T,
\end{align*}
where \(T=\bigO{E^3}.\)
\end{lemma}

\begin{proof}
Using the additive decomposition of the kernel matrix we employ a Neumann series \cite{neumann-encyc-maths} expansion of the inverse and defining \(Q:=(K^\infty)^{-1}\):
\eq{
        K^{-1} &= (E+K^\infty)^{-1} = (K^\infty)^{-1}[E(K^\infty)^{-1}+I]^{-1} \\
        &= Q[I - EQ + EQEQ] - T  \\
        &= 2Q  - Q KQ + Q KQ KQ  - 2Q KQ  + Q - T \\
        &= 3(K^\infty)^{-1}  - 3(K^\infty)^{-1} K(K^\infty)^{-1} + (K^\infty)^{-1} K(K^\infty)^{-1} K(K^\infty)^{-1} -T.  
    }
    which converges if \(\norm{EQ}<1\) for a suitable choice of matrix norm. We choose the matrix 1-norm for convenience.

    Firstly we have that \(Q=\frac{1}{\sigma_\xi^2}\left(I - \frac{\sigma_f^2}{\sigma_\xi^2+m\sigma_f^2}\1\1\right)\) from which we see that \(\norm{Q}_1 = \frac{1}{\sigma_\xi^2+m\sigma_f^2}\): 
    \eq{
        \norm{Q}_1 &= \max_i \sum_j \frac{1}{\sigma_\xi^2}\left(\delta_{ij} - \frac{\sigma_f^2}{\sigma_\xi^2+m\sigma_f^2}1_i1_j\right) \\
        &= \frac{1}{\sigma_\xi^2}\left(1 - \frac{m\sigma_f^2}{\sigma_\xi^2+m\sigma_f^2}\right) \\
        &= \frac{1}{\sigma_\xi^2+m\sigma_f^2}.
    }
    Now, note that \(\norm{E}_1 = \max_i\sum_j\eps_{ij}\leq m\sigma_f^2\) since \(\eps_{ij}\leq\sigma_f^2\) for all \(i,j\).
    Then,
    \eq{
        \norm{EQ}_1 &\leq \norm{E}_1\norm{Q}_1 \\
        &\leq \frac{m\sigma_f^2}{\sigma_\xi^2+m\sigma_f^2} \\
        & < 1
        }
        for all \(\sigma_\xi^2 > 0\) as desired.
\end{proof}

\begin{lemma}
\label{lem:1TKinv1}
    Under the same conditions as \ref{lem:Kinv},
\eq{
    \1^TK^{-1}\1 &= \sigma_f^{-2}\!\left(1+\frac{\avg{\eps_{\bullet\bullet}}-m^{-1}\sigma_\xi^2}{\sigma_f^2}\right) \!\pm\! \bigO{\eps^2+m^{-2}}
}
\end{lemma}
\begin{proof}
Using \ref{lem:Kinv} (neglecting the \(\bigO{E^3}\) term) and observing that \(\1^TK\1 = m^2\sigma_f^2 + m\sigma_\xi^2 - \sum_{i\neq j}\eps_{ij}\):
\eq{
    \1^TK^{-1}\1 &= 3m\sigma_\xi^{-2}(1-m\gamma) - 3\sigma_\xi^{-4}(1-m\gamma)^2\1^TK\1 + \sigma_\xi^{-4}(1-m\gamma)^2\1^TK(K^\infty)^{-1}K\1 \\
    &= 3\sigma_\xi^{-2}(1-m\gamma)(m-\sigma_\xi^{-2}(1-m\gamma)\1^TK\1) + \sigma_\xi^{-6}(1-m\gamma)^2\left[\1^TK^2\1 - \gamma(\1^TK\1)^2\right] \\
    &= \frac{3}{\sigma_f^2}\left(1-\frac{1}{m\sigma_f^2}(m\sigma_f^2+\sigma_\xi^2-m\avg{\eps_{\bullet\bullet}})\right) + \frac{(\1^TK\1)^2}{m^2\sigma_\xi^2\sigma_f^4}\left[m^{-1}-m^{-1}\left(1-\frac{\sigma_\xi^2}{m\sigma_f^2}\right)\right] \\
    &= \frac{3}{\sigma_f^4}\left(\avg{\eps_{\bullet\bullet}}-\frac{\sigma_\xi^2}{m}\right) + \frac{(m\sigma_f^2+\sigma_\xi^2-m\avg{\eps_{\bullet\bullet}})^2}{m^2\sigma_f^6} \\
    &= \frac{3}{\sigma_f^4}(\avg{\eps_{\bullet\bullet}}-m^{-1}\sigma_\xi^2) + \frac{1}{\sigma_f^2} - \frac{2}{\sigma_f^4}(\avg{\eps_{\bullet\bullet}}-m^{-1}\sigma_\xi^2) \pm \bigO{\eps^2+m^{-2}} \\
    &= \frac{1}{\sigma_f^2} + \frac{\avg{\eps_{\bullet\bullet}}-m^{-1}\sigma_\xi^2}{\sigma_f^4} \pm \bigO{\eps^2+m^{-2}}.
}
\end{proof}

\begin{lemma}
\label{lem:epsKinv1}
    Under the same conditions as \ref{lem:Kinv},
    \eq{
        \bm{\eps}^TK^{-1}\1 &= \frac{\avg{\eps_\bullet}}{\sigma_f^2}(1+\sigma_\xi^2/m\sigma_f^2)^{-1} + \bigO{\eps^2}.
    }
\end{lemma}
\begin{proof}
        Using \cref{lem:Kinv} and that \(Q^{-1}=K^\infty\) so \(Q\1 = \frac{1}{\sigma_\xi^2}\left(I - \frac{\sigma_f^2}{\sigma_\xi^2+m\sigma_f^2}\1\1\right) \bm{1} = \frac{\1}{\sigma_\xi^2+m\sigma_f^2}\) we have, neglecting the term \(T\) which will turn out to be negligible,
\begin{align}
    \bm{\eps}^TK^{-1}\1 &= 3\bm{\eps}^TQ\1 - 3\bm{\eps}^TQKQ\1 + \bm{\eps}^TQKQKQ\1 -\bm{\eps}^TT\1 \nonumber\\
    &= \frac{3m\avg{\eps_\bullet}}{\sigma_\xi^2+m\sigma_f^2} - \frac{3\bm{\eps}^TQK\1}{\sigma_\xi^2+m\sigma_f^2} + \frac{\bm{\eps}^TQKQK\1}{\sigma_\xi^2+m\sigma_f^2}. \label{eq:epsKinv1_0}
    \end{align}
    Next, recalling that \(E=K-K_\infty = K - Q^{-1}\) :
    \begin{align}
        \bm{\eps}^TQK\1 &= \bm{\eps}^T\1 + \bm{\eps}^TQE\1 \nonumber\\
        \bm{\eps}^TQKQK\1 &= \bm{\eps}^TQ(Q^{-1}+E)Q(Q^{-1}+E)\1 \nonumber\\
        &= \bm{\eps}^T(I+QE)(I+QE)\1 \nonumber\\
        &= \bm{\eps}^T\1 + 2\bm{\eps}^TQE\1 + \bm{\eps}^TQEQE\1 \label{eq:quad_E_term}.
    \end{align}
    
Now observe that \(-E\geq 0\) so that  \(-E\1=\bm{e}\geq 0\) (where both of these inequalities apply elementwise) . Also observe that \(Q\) is monotone, as defined in \cref{def:monotone} below, so that  by \cref{lem:mix_quad_prod_monotone} \(\bm{\eps}^TQ\bm{e}\geq 0\). Let \(\bm{q}=Q\bm{\eps}\). Then \(0
\leq\bm{q}^T\bm{e}=\sum_iq_ie_i \leq \sum_iq_ie_{\max}\) where \(e_{\max}=\norm{E}_1\). We use this fact immediately below:

\eq{
    \frac{-1}{\sigma_\xi^2 + m\sigma_f^2}\bm{\eps}^TQE\1 &\leq \frac{-1}{m\sigma_f^2}\bm{\eps}^TQE\1 \\
     &\leq \frac{1}{m\sigma_f^2}\bm{\eps}^TQ\1\norm{E}_1 \\
    &\leq \frac{1}{\sigma_f^4}\avg{
    \eps_\bullet}(3\eps_{\max}^\bullet + \avg{\eps_\bullet})(1+\sigma_\xi^2/m\sigma_f^2) \\
    &=\bigO{\eps^2}
}

where we have used 
\cref{lem:E1norm} in the penultimate line. We can apply the preceding argument iteratively to the third term in \eqref{eq:quad_E_term} to obtain a factor of \(m^{-1}\) and \(\norm{E}_1\) for each \(QE\) term, leading to the conclusion that the term is \(\bigO{\eps^3}\) in magnitude. 

Returning to \eqref{eq:epsKinv1_0} and noting that in \eqref{eq:quad_E_term} the term \(\bm{\epsilon}^T \1\) is equal to \(m\avg{\eps_\bullet}\), we have 
\eq{
    \bm{\eps}^TK^{-1}\1 &= \frac{1}{\sigma_\xi^2+m\sigma_f^2}(3m\avg{\eps_\bullet} -3m\avg{\eps_\bullet} - 3\bm{\eps}^TQE\1 + m\avg{\eps_\bullet} + 2\bm{\eps}^TQE\1 + \bigO{\eps^3})\nonumber\\
    &= \frac{m\avg{\eps_\bullet}}{\sigma_\xi^2 +m\sigma_f^2} + \bigO{\eps^2} \nonumber\\
    &= \frac{\avg{\eps_\bullet}}{\sigma_f^2}(1+\sigma_\xi^2/m\sigma_f^2)^{-1} + \bigO{\eps^2}.
}
\end{proof}

\begin{lemma}
\label{lem:E1norm}
    \(\norm{E}_1 \leq m(\eps_{\max}^\bullet + \avg{\eps_\bullet} + 2\sqrt{\eps_{\max}^\bullet}\sqrt{\avg{\eps_\bullet}})\).
\end{lemma}
\begin{proof}
We use the definition of the 1-norm in conjunction with the triangle inequality. Note that \(\eps_{ij}=\rho^2(\xrv_{(i)}(\xv),\xrv_{(j)}(\xv))\) and the triangle inequality implies \(\rho(\xrv_{(i)},\xrv_{(j)}) \leq \rho(\xrv_{(i)},\xv) + \rho(\xrv_{(j)},\xv)\). Hence, \(\eps_{ij} \leq \eps_i + \eps_j + 2\sqrt{\eps_i\eps_j}\). Then,
        \eq{
        \norm{E}_1 &= \max_i \sum_j \eps_{ij} \\
        &\leq \max_i \sum_j (\eps_i + \eps_j + 2\sqrt{\eps_i\eps_j}) \\
        &= m(\eps_{\max}^\bullet + \avg{\eps_\bullet}) + 2\sqrt{\eps_{\max}^\bullet}\sum_j\sqrt{\eps_j} \\
        &\leq m(\eps_{\max}^\bullet + \avg{\eps_\bullet}) + 2m\sqrt{\eps_{\max}^\bullet}\sqrt{\avg{\eps_\bullet}}.
    }
\end{proof}
\subsubsection{Proof of \texorpdfstring{\cref{lem:mse_conv_var}}{}}
\label{proof:mse_conv_var}
The proof follows by combining the upper bound \cref{lem:MSE_UB} with its expectation \cref{thm:mse_conv} and know that the best possible result is given by \(\sigma_\xi^2(1+m^{-1})\) (\cref{lem:min_mse}) and combine these effective upper and lower bounds by using the standard expression for variance.
    \begin{align*}
        \Var{f_n^{\MSE}(\bm{\theta})} &= \E{f_n^{\MSE}(\bm{\theta})^2} - \E{f_n^{\MSE}(\bm{\theta})}^2  \\
        &\leq \E{\left[\sigma_\xi^2 + \sigma_f^2 - \frac{\sigma_f^2 - 2\avg{\eps_{\bullet}}}{1+\sigma_\xi^2/m\sigma_f^2}\right]^2} - \sigma_\xi^4(1+m^{-1})^2 \\
        &\leq \E{\left[\sigma_\xi^2(1+m^{-1})+2\avg{\eps_{\bullet}}\right]^2} - \sigma_\xi^4(1+m^{-1})^2 \\
        &= 4\sigma_\xi^2(1+m^{-1})\E{\avg{\eps_{\bullet}}} + \bigO{\eps^2} \\
        &\leq 4\sigma_\xi^2(1 + m^{-1})C_{\bm{\theta},d}\left(\frac{n}{m}\right)^{-\frac{p}{d}}+ R_{mn}.
    \end{align*}
\hfill\(\square\)

\subsubsection{Proof of \texorpdfstring{\cref{lem:prob_mse_bound}}{}}
\label{proof:prob_mse_bound}
    We use Cantelli's inequality, \(\Pr{X-\mu\geq \lambda}\leq \frac{\sigma^2}{\sigma^2+\lambda^2}\) and note that in our case we only have results for upper and lower bounds on \(\mu\) and \(\sigma^2\) that we denote with bars and underlines respectively. 
    Let \(\mu_\Delta=\bar{\mu}-\underline{\mu}\). Hence, the quantity of interest can be bounded as
    \begin{align*}
        \Pr{X-\bar{\mu}\geq \lambda - \mu_\Delta} &= \Pr{X-\mu\geq \lambda - (\mu-\underline{\mu})} \\
        &\leq \Pr{X-\mu\geq \lambda - \mu_\Delta} \\
        &\leq \frac{\bar{\sigma}^2}{\underline{\sigma}^2 + [\lambda - (\bar{\mu}-\underline{\mu})]^2}
    \end{align*}
    where the last line follows from Cantelli's inequality. 
    
    Now in our case we refer to \cref{lem:mse_conv_var} for \(\bar{\sigma}^2\) and note that the lower bound is 0. For \(\bar{\mu}\) we use \cref{lem:MSE_UB} and for \(\underline{\mu}\) we use the minimum value \(\sigma_\xi^2\). Finally, choose \(\lambda=\varepsilon\). Hence,
    \begin{align}
    \label{eq:prob_mse_bound}
        \Pr{f_n^{\MSE}(\bm{\theta}) \geq \varepsilon + \underline{\mu}} &\leq \frac{\bar{\sigma}^2}{[\varepsilon - (\bar{\mu}-\underline{\mu})]^2}.
    \end{align}

    The RHS of \eqref{eq:prob_mse_bound} is given by 
    
    \eq{
        &\frac{4\sigma_\xi^2(1+m^{-1})C_{\bm{\theta},d}(n/m)^{-p/d} + R_{mn}}{(\varepsilon-2C_{\bm{\theta},d}(n/m)^{-p/d} -R_{mn})^2},
    }
        which we then bound from above by a constant \(\delta>0\). Let \(u=2C_{\bm{\theta},d}(n/m)^{-p/d}<1\) and \(A=2\sigma_\xi^2(1+m^{-1})\) then we have that \(\frac{Au}{(\varepsilon-u )^2}<\delta\) (neglecting terms in \(R_{mn}\)) and thus \(u^2-[2(\varepsilon)+A/\delta]u + \varepsilon^2 >0\). Solving the quadratic to first order we obtain \(u_1=-\varepsilon < 0\) and \(u_2=3\varepsilon+A/\delta-\bigO{\varepsilon^2}\), and noting that \(0<u<1\), we use the positive solution \(u_2\) to get that \(u<u_2\).
    
    Substituting the original values and rearranging gives 
    \eq{
        n &> m\left[\frac{3\varepsilon}{2C_{\bm{\theta},d}}+\frac{\sigma_\xi^2(1+m^{-1})}{C_{\bm{\theta},d}\delta}\right]^{-\frac{d}{p}} + \bigO{\varepsilon^2}.\\
    }
    \hfill\(\square\)

%% file: appendices/misspec_rates.tex
\section{Convergence rates in the misspecified case}

Before deriving convergence rates in the misspecifed case we reproduce the expansion given in Lemma 16 A.2.1 of \cite{GPnn} of the MSE,
\begin{align*}
    &\Ey{\Eystar{(y^* - \hat{\mu}_N^*)^2 \given \yrv}} \\
    &= \Ey{\Eystar{{y^*}^2 - 2y^*\hat{\mu}_N^* + (\hat{\mu}_N^*)^2 \given \yrv}} \\
    &= \Ey{\predvar_N + \mu^{*\,2}_N - 2\mu^*_N\hat{\mu}_N^* + (\hat{\mu}_N^*)^2} \\
    &= \underbrace{\predvar_N}_{(a)} + \underbrace{\kstar_N^TK_N^{-1}{\kstar_N}}_{(b)} - 2\underbrace{\kstar_N^T\hat{K}_N^{-1}\kstar[\hat]_N}_{(c)} + \underbrace{\kstar[\hat]_N^T\hat{K}_N^{-1}K_N\hat{K}_N^{-1}\kstar[\hat]_N}_{(d)}.
    \end{align*}
so we have that \(f_n^{\MSE}(\thv) = (a)+(b)-2(c)+(d)\). This is used several times in the sections that follow, although hereafter we remove the \(N\) subscript and \(*\) superscript. This is primarily to avoid overcrowding the notation, but also because the results given could be applied more generally provided the \(\eps\) terms are sufficiently small for relevant expansions to converge.

\subsection{Convergence Rates for Misspecification Type (a) or (b)}
Recall that \(f_n^{\MSE}(\thv) = (a)+(b)-2(c)+(d)\) and also that \((a)+(b)=\sigma_f^2+\sigma_\xi^2\).
Now we can analyse each of these terms using the expansion and the corresponding matrix computations but first some intermediate results are required.

\begin{lemma}
\label{lem:1TK21}
For a symmetric positive-definite \(m\times m\) real matrix \(K\) with constant diagonal entries and positive off-diagonal terms given by \(\sigma_f^2-\eps_{ij}\) for ``small'' \(\eps_{ij}>0\),
    \begin{equation*}
        \1^TK^2\1 = m^{-1}(\1^TK\1)^2 \pm \bigO{\eps^2}.
    \end{equation*}
\end{lemma}
\begin{proof}
    For diagonal terms given by \(\sigma_f^2+\sigma_\xi^2\):
    \eq{
    \1^TK\1 &= n^2\sigma_f^2 + n\sigma_\xi^2 - \sum_{i\neq j}\eps_{ij}, \\
    (\1^TK\1)^2 &= n^4\sigma_f^4 + n^2\sigma_\xi^4 + \sum_{i\neq j}\eps_{ij}\sum_{k\neq l}\eps_{kl} + 2n^3\sigma_f^2\sigma_\xi^2 - 2n\sigma_\xi^2\sum_{i\neq j}\eps_{ij} - 2n^2\sigma_f^2\sum_{i\neq j}\eps_{ij}, \\
    \1^TK^2\1 &= n^3\sigma_f^4 + 2n^2\sigma_f^2\sigma_\xi^2 + n\sigma_\xi^4 - 2n^2\avg{\eps_{\bullet\bullet}} -2n\sigma_\xi^2\avg{\eps_{\bullet\bullet}} \pm\bigO{\eps^2}.
    }
    Using that \(\1^T\hat{K}K\1=\1^TK\hat{K}\1\)
and
\eq{
    \sum_{ijk} \hat{K}_{ij}K_{jk} &= \sum_{ijk}(\hat{\sigma}_f^2+\delta_{ij}\hat{\sigma}_\xi^2-(1-\delta_{ij}\hat{\eps}_{ij})) (\sigma_f^2+\delta_{jk}\sigma_\xi^2-(1-\delta_{jk}\eps_{jk})) \\
    &= n^3\hat{\sigma}_f^2\sigma_f^2 + \hat{\sigma}_f^2\sigma_\xi^2\sum_{ijk}\delta_{jk} d- \hat{\sigma}_f^2\sum_{ijk}(1-\delta_{jk})\eps_{jk} + \sigma_f^2\hat{\sigma}_\xi^2\sum_{ijk}\delta_{ij} \\
    &\quad- \sigma_f^2\sum_{ijk}(1-\delta_{ij})\hat{\eps}_{ij}+\sum_{ijk}\delta_{ij}\delta_{jk}\sigma_\xi^2\hat{\sigma}_\xi^2-\sum_{ijk}\delta_{ij}(1-\delta_{jk})\hat{\sigma}_\xi^2\eps_{jk} \\
    &\quad- \sum_{ijk}\delta_{jk}(1-\delta_{ij})\sigma_\xi^2\hat{\eps}_{ij}+ \sum_{ijk}(1-\delta_{ij})(1-\delta_{jk})\eps_{jk}\hat{\eps}_{ij} \\
    &= n^3\sigma_f^2\hat{\sigma}_f^2 + n^2(\sigma_\xi^2\hat{\sigma}_f^2+\hat{\sigma}_\xi^2\sigma_f^2) + n\sigma_\xi^2\hat{\sigma}_\xi^2 - n^2\hat{\sigma}_f^2\avg{\eps_{\bullet\bullet}} -n^2\sigma_f^2\avg{\hat{\eps}_{\bullet\bullet}}\\
    &\quad -n\hat{\sigma}_\xi^2\avg{\eps_{\bullet\bullet}} -n\sigma_\xi^2\avg{\hat{\eps}_{\bullet\bullet}} \pm\bigO{\eps^2}.
}
\end{proof}

We can use these results, in combination with those derived in \cref{appx:rates}, to provide an upper bound for the predictive MSE under model misspecification. In this way we obtain \cref{thm:mse_conv_mis}, having first proved \cref{lem:misspec_MSE_UB}:

\subsubsection{Proof of \texorpdfstring{\cref{lem:misspec_MSE_UB}}{}}
\label{proof:misspec_MSE_UB}
Using the decomposition given in \cite{GPnn} we then use the results just derived (\ref{lem:1TKinv1}) plus an additional computation given below.

Using \cref{lem:Kinv} but retaining only first order terms in \(\hat{E}\) from the Neumann expansion due to the fact that this term is quadratic, we have (defining \(\hat{Q}\1=\hat{q}\1\) with \(\hat{q}=\hat{\sigma}_\xi^{-2}(1-m\hat{\gamma})\))
\eq{
     \1^T\hat{K}^{-2}\1 &= \1^T\hat{Q}[I-\hat{E}\hat{Q}]\hat{Q}[I-\hat{E}\hat{Q}]\1 + \bigO{\eps^2} \\
     &= \hat{q}\1^T[\hat{Q}-\hat{E}\hat{Q}^2-\hat{Q}\hat{E}\hat{Q}]\1 + \bigO{\eps^2} \\
     &= \hat{q}^2(m-2\hat{q}
     \1^T\hat{E}\1) + \bigO{\eps^2} \\
     &= \hat{\sigma}_\xi^{-4}(m-2\hat{\sigma}_\xi^{-2}m(m-1)\avg{\hat{\eps}_{\bullet\bullet}}) + \bigO{\eps^2} \\
     &= \frac{1}{m\hat{\sigma}_f^4} - \frac{2\avg{\hat{\eps}_{\bullet\bullet}}}{m\hat{\sigma}_f^4} + \bigO{\eps^2+m^{-2}\eps + m^{-2}} \\
     &\leq \frac{1}{m\hat{\sigma}_f^4}.
}


\paragraph{(c)}
Using \cref{lem:1TKinv1} and \cref{lem:epsKinv1}:
\eq{
        \bm{k}^T\hat{K}^{-1}\hat{\bm{k}} &= (\sigma_f^2\1-\bm{\eps})^T\hat{K}^{-1}(\hat{\sigma}_f^2\1-\hat{\bm{\eps}}) \\
        &= \sigma_f^2\hat{\sigma}_f^2\1^T\hat{K}^{-1}\1 - \sigma_f^2\hat{\bm{\eps}}^T\hat{K}^{-1}\1 - \hat{\sigma}_f^2\bm{\eps}^T\hat{K}^{-1}\1 + \bm{\eps}^T\hat{K}^{-1}\hat{\bm{\eps}} \\
        &= \sigma_f^2 + \frac{\sigma_f^2}{\hat{\sigma}_f^2}(\avg{\hat{\eps}_{\bullet\bullet}} - m^{-1}\hat{\sigma}_\xi^2) - \frac{\frac{\sigma_f^2}{\hat{\sigma}_f^2}\avg{\hat{\eps}_\bullet} + \avg{\eps_\bullet}}{1+\hat{\sigma}_\xi^2/m\hat{\sigma}_f^2} \pm \bigO{\eps^2+m^{-2}}.
}

\paragraph{(d)}

Since \(K_{jk}=\sigma_f^2-\eps_{jk}+\delta_{jk}\sigma_\xi^2\) where \(\eps_{jj}=0\), we have that
\eq{
    \1^T\hat{K}^{-1}K\hat{K}^{-1}\1 &= \sum_{ijkl}(\hat{K}^{-1})_{ij}K_{jk}(\hat{K}^{-1})_{kl} \\
    &= \sigma_f^2\sum_{ijkl}(\hat{K}^{-1})_{ij}(\hat{K}^{-1})_{kl} + \sigma_\xi^2\sum_{ijk}(\hat{K}^{-1})_{ij}(\hat{K}^{-1})_{jk} - \sum_{ijkl}(\hat{K}^{-1})_{ij}(\hat{K}^{-1})_{kl}\eps_{jk} \\
    &= \sigma_f^2(\1^T\hat{K}^{-1}\1)^2 + \sigma_\xi^2\1^T\hat{K}^{-2}\1+ \1^T\hat{K}^{-1}E\hat{K}^{-1}\1
}

Observe that \(\1^T\hat{K}^{-1}E\hat{K}^{-1}\1\leq 0\) since \(\hat{K}^{-1}\) is monotone (\cref{def:monotone}) so with \(\bm{u}=\hat{K}^{-1}\1\geq\0\), \(\bm{u}^TE\bm{u}=-\sum_{ij}u_iu_j\eps_{ij}\leq 0\).

Now subbing in our intermediates:
\eq{
    \1^T\hat{K}^{-1}K\hat{K}^{-1}\1 &= \sigma_f^2\hat{\sigma}_f^{-4}\left(1+\hat{\sigma}_f^{-2}(\avg{\hat{\eps}_{\bullet\bullet}} - \frac{\hat{\sigma}_\xi^2}{m})\right)^2 + \sigma_\xi^2\frac{1}{m\hat{\sigma}_f^4} + \1^T\hat{K}^{-1}E\hat{K}^{-1}\1 \\
    &\leq \hat{\sigma}_f^{-4}\left(\sigma_f^2 + \frac{\sigma_\xi^2}{m} + \frac{2\sigma_f^2}{\hat{\sigma}_f^2}(\avg{\hat{\eps}_{\bullet\bullet}}-\frac{\hat{\sigma}_\xi^2}{m})\right) + \bigO{\eps^2+m^{-2}}
}

and 
\eq{
    \hat{\bm{k}}^T\hat{K}^{-1}K\hat{K}^{-1}\hat{\bm{k}} &= (\hat{\sigma}_f^2\1-\hat{\bm{\eps}})^T\hat{K}^{-1}K\hat{K}^{-1}(\hat{\sigma}_f^2\1-\hat{\bm{\eps}}) \\
    &= \hat{\sigma}_f^4\big(\sigma_f^2(\1^T\hat{K}^{-1}\1)^2 + \sigma_\xi^2\1^T\hat{K}^{-2}\1 + \1^T\hat{K}^{-1}E\hat{K}^{-1}\1\big) \\
    &- 2\hat{\bm{\eps}}^T\big(\sigma_f^2\hat{K}^{-1}\1\1^T\hat{K}^{-1}\1 + \sigma_\xi^2\hat{K}^{-2}\1 + \hat{K}^{-1}E\hat{K}^{-1}\1\big) + \bigO{\eps^2}
}
so,
\eq{
    \hat{\bm{k}}^T\hat{K}^{-1}K\hat{K}^{-1}\hat{\bm{k}} &\leq (\hat{\sigma}_f^2-\hat{\eps}^{\bullet}_{\min})^2\1^T\hat{K}^{-1}K\hat{K}^{-1}\1 \\
    &\leq \sigma_f^2 + \frac{\sigma_\xi^2}{m} + \frac{2\sigma_f^2}{\hat{\sigma}_f^2}\big(\avg{\hat{\eps}_{\bullet\bullet}}-\frac{\hat{\sigma}_\xi^2}{m}\big) - \frac{2\hat{\eps}^{\bullet}_{\min}}{\hat{\sigma}_f^2}\bigg[\sigma_f^2 + \frac{\sigma_\xi^2}{m}+ \frac{2\sigma_f^2}{\hat{\sigma}_f^2}\bigg(\avg{\hat{\eps}_{\bullet\bullet}} -\frac{\hat{\sigma}_\xi^2}{m}\bigg)\bigg]
}
and finally,
\eq{
    (a)+(b)-2&(c)+(d) \\
    &\leq \sigma_f^2 + \sigma_\xi^2 - 2\sigma_f^2\left(1+\frac{\avg{\hat{\eps}_{\bullet\bullet}} - m^{-1}\hat{\sigma}_\xi^2}{\hat{\sigma}_f^2}\right) + 2\frac{\sigma_f^2}{\hat{\sigma}_f^2}\avg{\hat{\eps}_{\bullet}} + 2\avg{\eps_{\bullet}} + \sigma_f^2 \\
    &+ 2\frac{\sigma_f^2}{\hat{\sigma}_f^2}(\avg{\hat{\eps}_{\bullet\bullet}}-m^{-1}\hat{\sigma}_\xi^2) + \frac{\sigma_\xi^2}{m} - 2\hat{\eps}^{\bullet}_{\min}\frac{\sigma_f^2}{\hat{\sigma}_f^2} + \bigO{\eps^2+m^{-2}} \\
    &= \sigma_\xi^2(1+m^{-1}) + 2\frac{\sigma_f^2}{\hat{\sigma}_f^2}(\avg{\hat{\eps}_{\bullet}}-\hat{\eps}^{\bullet}_{\min})+ 2\avg{\eps_{\bullet}} \pm \bigO{\eps^2+m^{-2}}.
}
\hfill\(\square\)

\subsubsection{Proof of \texorpdfstring{\cref{thm:mse_conv_mis}}{}}
\label{proof:mse_conv_mis}
The result follows from \cref{lem:misspec_MSE_UB} and observing that \(\hat{\eps}^{\bullet}_{\min}>0\) so that \(f_n^{\MSE}(\thv)\leq\sigma_\xi^2(1+m^{-1})+2\frac{\sigma_f^2}{\hat{\sigma}_f^2}\avg{\hat{\eps}_{\bullet}} + 2\avg{\eps_{\bullet}}\). We then apply \cref{lem:eps_i} as per \cref{thm:mse_conv} to \(\avg{\eps_{\bullet}}\) and identically to \(\avg{\hat{\eps}_{\bullet}}\) but with \(\hat{p}\) and \(C_{\thv,d}\) replacing \(p\) and \(C_{\bm{\theta},d}\).
\hfill\(\square\)

\subsubsection{Proof of \texorpdfstring{\cref{lem:mse_conv_var_mis}}{}}
\label{proof:mse_conv_var_mis}
As with \cref{lem:mse_conv_var}, we have
    \begin{align*}
        \Var{f_n^{\MSE}(\thv)} &= \E{f_n^{\MSE}(\thv)^2} - \E{f_n^{\MSE}(\thv)}^2 \\
        & \leq \Eold\bigg[\sigma_\xi^2(1+m^{-1}) + 2\frac{\sigma_f^2}{\hat{\sigma}_f^2}\avg{\hat{\eps}_\bullet} + 2\avg{\eps_\bullet}  \pm \bigO{\eps^2+m^{-2}}\bigg]^2 - \sigma_\xi^4(1+m^{-1})^2 \\
        & = 2\sigma_\xi^2(1+m^{-1})\E{\frac{\sigma_f^2}{\hat{\sigma}_f^2}\avg{\hat{\eps}_\bullet}+\avg{\eps_\bullet}} + \bigO{\eps^2+m^{-2}} \\
         & \leq 2\tilde{C}^m_{\bm{\theta},d}\left(\frac{n}{m}\right)^{-\frac{p}{d}}+ 2\tilde{C}^m_{\thv,d}\left(\frac{n}{m}\right)^{-\frac{\hat{p}}{d}} + R_{mn},
    \end{align*}
    with \(\tilde{C}^m_{\thv,d}=2\sigma_\xi^2(1+m^{-1}) \frac{\sigma_f^2}{\hat{\sigma}_f^2} C_{\thv,d}\) and \(\tilde{C}^m_{\bm{\theta},d}\) as before.
    
\hfill\(\square\)

\begin{lemma}[Probabilistic MSE guarantees (under misspecification)]
\label{lem:prob_mse_bound_mis}
    Under the same conditions as \cref{thm:mse_conv_mis},
    for constants \(\varepsilon,\delta > 0\), 
    if \(\hat{p}\geq p\) and \(n\left(1+\bigO{n^{\frac{p-\hat{p}}{d}}}\right)^{-\frac{d}{p}}\gtrsim m\left[\frac{3\varepsilon}{2C_{\bm{\theta},d}}+\frac{\sigma_\xi^2(1+m^{-1})}{C_{\bm{\theta},d}\delta}\right]^{-\frac{d}{p}}\) then
    with probability at least \(1-\delta\),
    \[f_n^{\MSE}(\thv) < \sigma_\xi^2(1+m^{-1}) + \varepsilon.\]
    If \(\hat{p}\leq p\) we can simply interchange the roles of \(p\) and \(\hat{p}\) in the expressions.
\end{lemma}
\begin{proof}
    The proof follows in identical fashion to \cref{lem:prob_mse_bound} but with the upper and lower bounds substituted into Cantelli's inequality taken from \cref{thm:mse_conv_mis} and \cref{lem:mse_conv_var_mis}. To obtain the LHS of the condition on \(n\) we use the same methods as in \cref{lem:prob_mse_bound} only now \(u=2\left[C_{\bm{\theta},d}\left(\frac{n}{m}\right)^{-\frac{p}{d}}+\frac{\sigma_f^2}{\hat{\sigma}_f^2}C_{\thv,d}\left(\frac{n}{m}\right)^{-\frac{\hat{p}}{d}}\right]\). If we now assume \(\hat{p}\geq p\) then equivalently \(u=2C_{\bm{\theta},d}\left(\frac{n}{m}\right)^{-\frac{p}{d}}\left[1+\frac{\sigma_f^2}{\hat{\sigma}_f^2}\frac{C_{\thv,d}}{C_{\bm{\theta},d}}\left(\frac{n}{m}\right)^{\frac{p-\hat{p}}{d}}\right]\).
\end{proof}

\subsection{Calibration Convergence for Misspecification Type (a)(iii) or (b)}
\subsubsection{Proof of \texorpdfstring{\cref{lem:cal_conv}}{}}
\label{proof:cal_conv}
    Since \(f_n^{\CAL}(\thv)=\frac{f_n^{\MSE}(\thv)}{\predvar[\hat]_N}\) we first observe that the lower bound for the numerator is \(\sigma_\xi^2(1+m^{-1})\) (\cref{lem:min_mse}). Now we bound the denominator from above by the following:
    \eq{
        \predvar[\hat]_N &= \hat{\sigma}_f^2 + \hat{\sigma}_\xi^2 - \hat{\kappa}_N \leq \hat{\sigma}_\xi^2 + \frac{\hat{\sigma}_\xi^2/m+2\avg{\hat{\eps}_\bullet}}{1+\hat{\sigma}_\xi^2/(m\hat{\sigma}_f^2)}\\
        &\leq \hat{\sigma}_\xi^2(1+m^{-1}) + 2\avg{\hat{\eps}_\bullet} 
    }
    using \cref{lem:kappa}. The result follows.

    A lower bound for \(\predvar[\hat]_N\) is obtained analogously to \cref{lem:kappa}:
    \eq{
        \hat{\kappa}_N & \leq (\hat{\sigma}_f^2-\hat{\eps}^{\bullet}_{\min})^2\hat{\sigma}_f^{-2}\left(1-\frac{\avg{\hat{\eps}_{\bullet\bullet}}-m^{-1}\hat{\sigma}_\xi^2}{\hat{\sigma}_f^2}\right) \pm\bigO{\eps^2+m^{-2}} \qquad\textrm{(using \cref{lem:1TKinv1})}\\
        &= \hat{\sigma}_f^2 -2\hat{\eps}^{\bullet}_{\min}-\avg{\hat{\eps}_{\bullet\bullet}} + m^{-1}\hat{\sigma}_\xi^2 \pm \bigO{\eps^2+m^{-2}}, \\
        \predvar[\hat]_N &\geq \hat{\sigma}_f^2 + \hat{\sigma}_\xi^2 - (\hat{\sigma}_f^2 -2\hat{\eps}^{\bullet}_{\min}-\avg{\hat{\eps}_{\bullet\bullet}} + m^{-1}\hat{\sigma}_\xi^2 \pm \bigO{\eps^2+m^{-2}}) \\
        &= \hat{\sigma}_\xi^2(1+m^{-1}) + 2\hat{\eps}^{\bullet}_{\min} + \avg{\hat{\eps}_{\bullet\bullet}}\mp \bigO{\eps^2+m^{-2}}.
    }
    
    We combine this lower bound with \cref{lem:misspec_MSE_UB} to obtain the calibration upper bound.
\hfill\(\square\)

%% file: appendices/more_results.tex
\subsection{Additional Results}
\subsubsection{Proof of \texorpdfstring{\cref{lem:mse_limitl}}{}}
\label{proof:mse_limitl}
Using the same decomposition as in \cref{proof:misspec_MSE_UB} and applying the limiting behaviour of \(\hat{l}\) so that \(\hat{K}\rightarrow\hat{K}^\infty\) and \(\hat{\bm{k}}\rightarrow\hat{\sigma}_f^2\) we find:
\eq{
    (a) + (b) - &2(c) +(d) \\
    &= \sigma_f^2 + \sigma_\xi^2 -2\hat{\bm{k}}^T(\hat{K}^\infty)^{-1}\bm{k}  + \hat{\bm{k}}^T(\hat{K}^\infty)^{-1}K(\hat{K}^\infty)^{-1}\hat{\bm{k}} \\
    &= \sigma_f^2 + \sigma_\xi^2 -2 \frac{\hat{\sigma}_f^2}{\hat{\sigma}_\xi^2}(1-m\hat{\gamma})\1^T\bm{k} + \frac{\hat{\sigma}_f^4}{\hat{\sigma}_\xi^4}(1-m\hat{\gamma})^2\1^TK\1 \\
    &= \sigma_f^2 + \sigma_\xi^2 -2(\sigma_f^2 - \avg{\eps_\bullet}) + m^{-2}(m^2\sigma_f^2 + m\sigma_\xi^2 -\sum_{i\neq j}\eps_{ij}) \pm \bigO{m^{-2}} \\
    &= \sigma_\xi^2(1+m^{-1}) + 2\avg{\eps_\bullet} - \avg{\eps_{\bullet\bullet}} + \bigO{m^{-2}}
}
to which we apply \cref{lem:eps_i} to obtain the result.

\hfill\(\square\)

\subsubsection{High-Dimensions}
\label{appx:highd}
As \(d\) grows, the off-diagonal terms of the kernel matrix tend to constant values that match their expected values (for that \(d\)) and the fluctuations in these entries tends to 0. As these fluctuations tend to 0, the off-diagonal terms tend to constant values (that need not be 0 or \(\sigma_f^2\)) and behaviour tends to that of straight \(k\)-NN in a similar way to that derived in \cref{lem:mse_limitl}. This effect is magnified in the GPnn case as the nearest neighbour kernel entries tend to constant values as \(n\) grows too.

We can see this by considering the simple example of a Gaussian input measure and a kernel satisfying \ref{ass:kernel_dist}:
\begin{example}[Fluctuation example]
    \label{ex:fluctuations}
    Assume a Gaussian measure over the input space: \(\xrv \sim \mc{N}(0,\frac{1}{d}I_d)\) and a kernel satisfying \ref{ass:kernel_dist} for \(p=2\) and fixed \(l>0\). Then,
    \eq{
        \Pr{\abs{\rho^2(\Delta)-\E{\rho^2(\Delta)}}\geq \sigma_f^2\varepsilon} &\leq \frac{8L^2}{\varepsilon^2 dl^4} \xrightarrow{d\rightarrow\infty}0.
    }
\end{example}
\begin{proof}
    With the input distribution as assumed, we have that \(\norm{\xv-\xv'}^2\cdot\frac{d}{2}\sim\chi^2_d\). Hence, \(\Var{\Delta^2}=\frac{8}{d}\). Applying \ref{ass:kernel_dist} we have \(\Var{\rho^2(\Delta)/\sigma_f^2}\leq \frac{L^2}{l^4}\frac{8}{d}\). Finally, applying Chebyshev's inequality: \(\Pr{\abs{\rho^2(\Delta)/\sigma_f^2-\E{\rho^2(\Delta)/\sigma_f^2}}\geq\varepsilon}\leq\frac{\Var{\rho^2(\Delta)/\sigma_f^2}}{\varepsilon^2}\leq\frac{8L^2}{d\varepsilon^2l^4}\).
\end{proof}

This example can be generalised to the following lemma:
\begin{lemma}[Fluctuations with \(d\)]
    For a kernel satisfying \(\rho^2(\Delta)\leq L\left(\frac{\Delta}{l}\right)^p\) and an input distribution satisfying \(\Var{\Delta^p}=\bigO{d^\alpha}<\infty\) for \(\alpha \in \mathbb{R}\), then \(\rho^2(\Delta)\xrightarrow{d\rightarrow\infty}\E{\rho^2(\Delta)}\) in probability if \(d^\alpha/l(d)^{2p}\xrightarrow{d\rightarrow\infty}0\).
\end{lemma}
\begin{proof}
    The proof follows the same logic as \cref{ex:fluctuations} by applying Chebyshev's inequality directly to the assumptions given.
\end{proof}

We believe this relationship between lengthscale and dimensionality in part explains the superior performance of GPnn (empirically) over \(k\)-NN over an interval of \(\hat{l}\) around \(l\); that in this region the GP can exploit the correlations between \(y\) values and interpolate effectively. When \(\hat{l}\) is too large, as previously discussed in \cref{lem:mse_limitl}, behaviour tends to that of \(k\)-NN, whilst for \(\hat{l}\) too small we lose the informative contribution of the training points (since as \(\hat{l}\rightarrow0\) \(c(\xv,\xv')\rightarrow0\) for \(\xv\neq\xv'\)).

%% file: appendices/experiments.tex
\section{Further Information on Simulations and Experiments}
\label{appx:experiments}
Simulations were run as per Algorithm 1 \cite{GPnn} with 10 choices of \(n\) taken so that \(\log_{10}(n) \in [3,7]\), \((\sigma_f^2,\sigma_\xi^2)=(0.9,0.1)\), \(m=400\) and other parameters as shown in the relevant figures. Each simulation was repeated 5 times with different random seeds and the results averaged. Test size \(n^*\) was set at 1000. We use the L2 metric to obtain the distance matrix used to ascertain the nearest-neighbours, observing that, for the kernels used, this should generate equivalent neighbour sets as those we would obtain using the kernel-induced metric.

\subsection{Additional Simulations}
\label{appx:sims}
We ran each combination of kernel misspecification outlined in \ref{misspecs} ((b)(i) and (ii)) for values of \(n\) as outlined above, comparing \(d=15\) to \(d=50\) and to performance of a straight \(k\)-NN prediction (taking only the best \(k\)-NN performance over \(d\) and \(n\) for readability). Figures \ref{fig:conv_matched_rbf_rbf}, \ref{fig:conv_matched_exp_exp}, \ref{fig:conv_matched_rbf_exp} and \ref{fig:conv_matched_exp_rbf} show the MSE (not gradient) results for \(c_{RBF}\)-\(\hat{c}_{RBF}\), \(c_{Exp}\)-\(\hat{c}_{Exp}\), \(c_{RBF}\)-\(\hat{c}_{Exp}\) and \(c_{Exp}\)-\(\hat{c}_{RBF}\) respectively.

\begin{figure}[ht]
\vskip 0.2in
\begin{center}
\centerline{\includegraphics[width=\columnwidth]{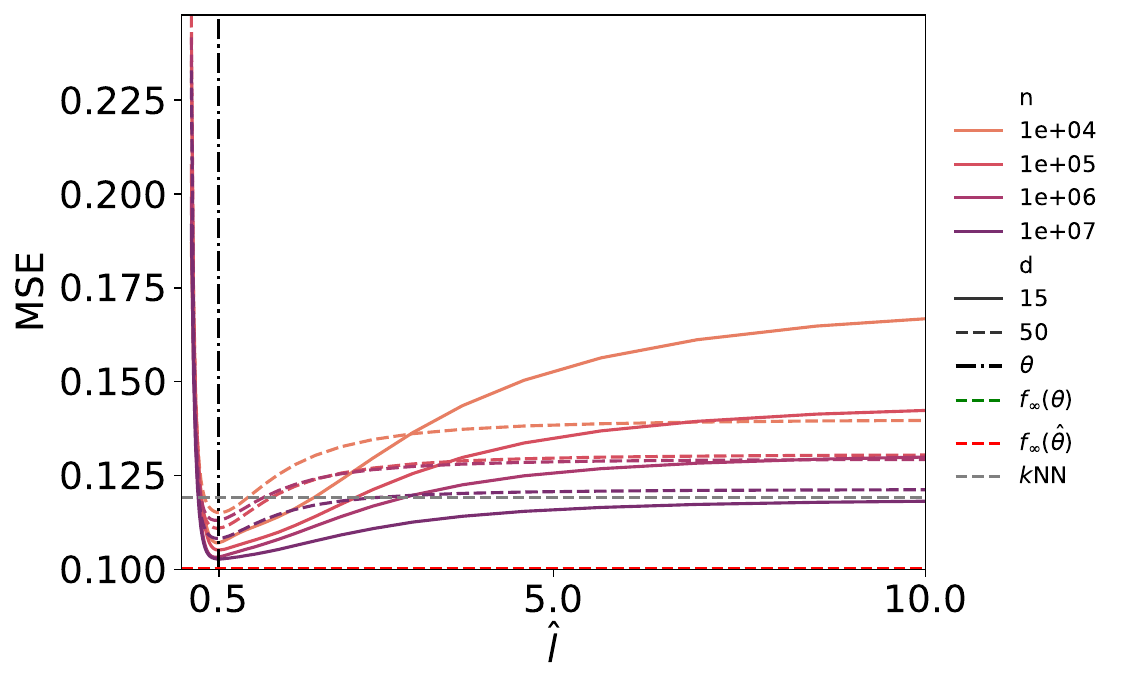}}
\caption{Convergence as a function of lengthscale with matched external parameters (\(c_{RBF}\)-\(\hat{c}_{RBF}\)) for variable \(n,d\). Horizontal grey dashed line represents the best MSE attained by a straight \(k\)-NN prediction (over both dimensions and all \(n\)).}
\label{fig:conv_matched_rbf_rbf}
\end{center}
\vskip -0.2in
\end{figure}

\begin{figure}[ht]
\vskip 0.2in
\begin{center}
\centerline{\includegraphics[width=\columnwidth]{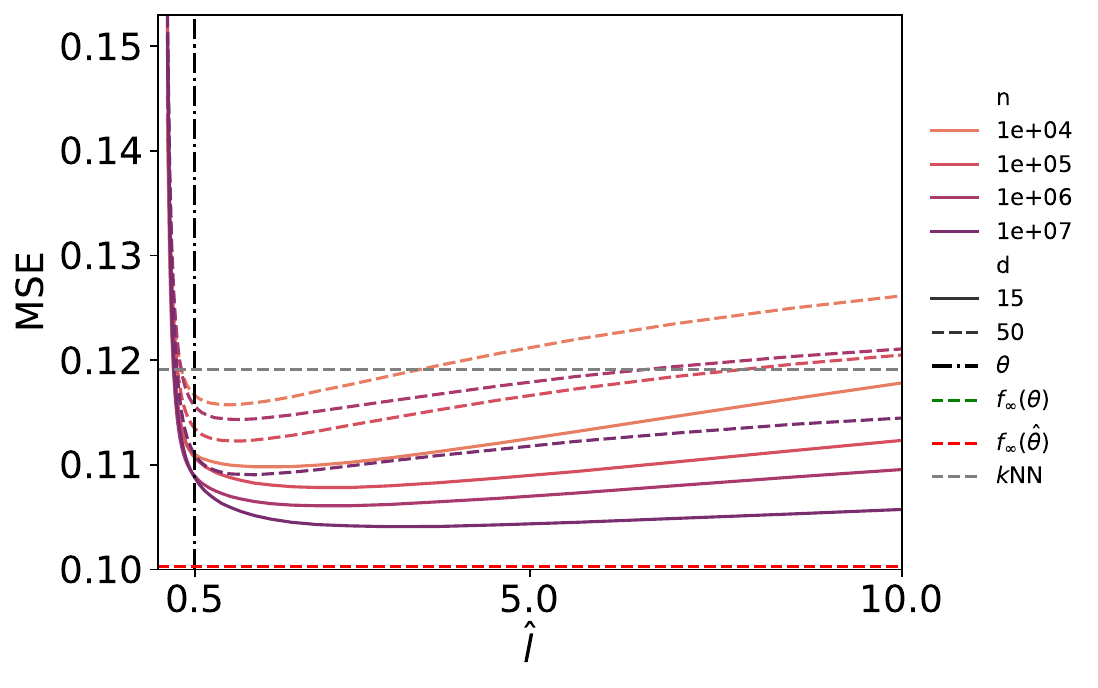}}
\caption{Convergence as a function of lengthscale with matched external parameters (\(c_{RBF}\)-\(\hat{c}_{Exp}\)) for variable \(n,d\). Horizontal grey dashed line represents the best MSE attained by a straight \(k\)-NN prediction (over both dimensions and all \(n\)).}
\label{fig:conv_matched_rbf_exp}
\end{center}
\vskip -0.2in
\end{figure}

\begin{figure}[ht]
\vskip 0.2in
\begin{center}
\centerline{\includegraphics[width=\columnwidth]{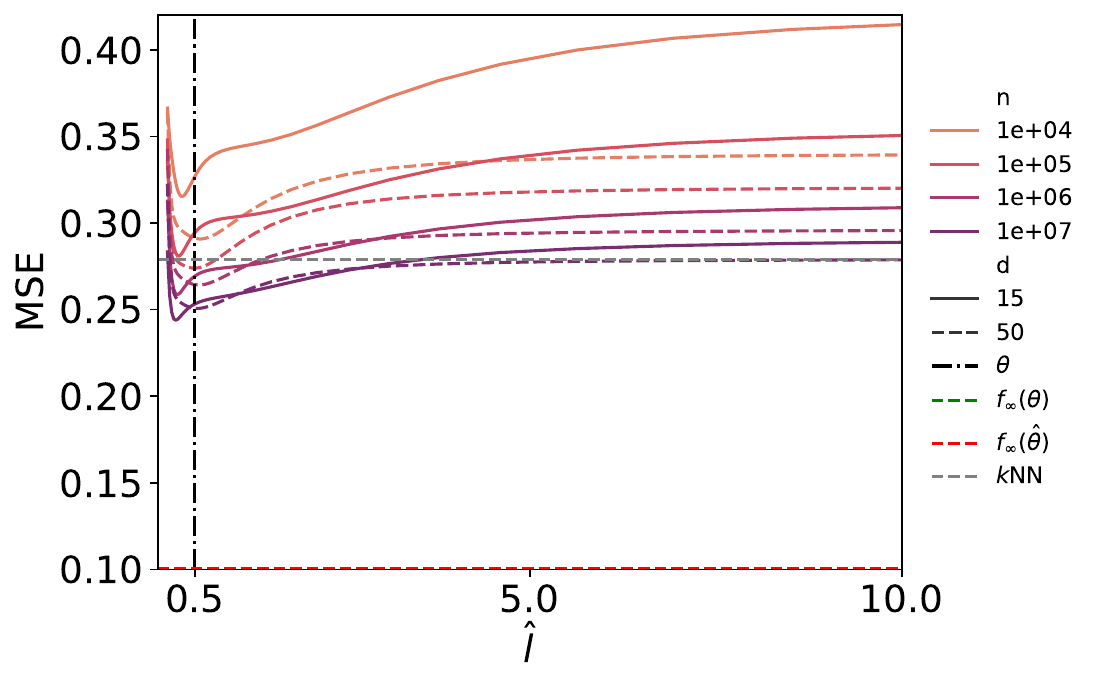}}
\caption{Convergence as a function of lengthscale with matched external parameters (\(c_{Exp}\)-\(\hat{c}_{RBF}\)) for variable \(n,d\). Horizontal grey dashed line represents the best MSE attained by a straight \(k\)-NN prediction (over both dimensions and all \(n\)).}
\label{fig:conv_matched_exp_rbf}
\end{center}
\vskip -0.2in
\end{figure}

\begin{figure}[ht]
\vskip 0.2in
\begin{center}
\centerline{\includegraphics[width=\columnwidth]{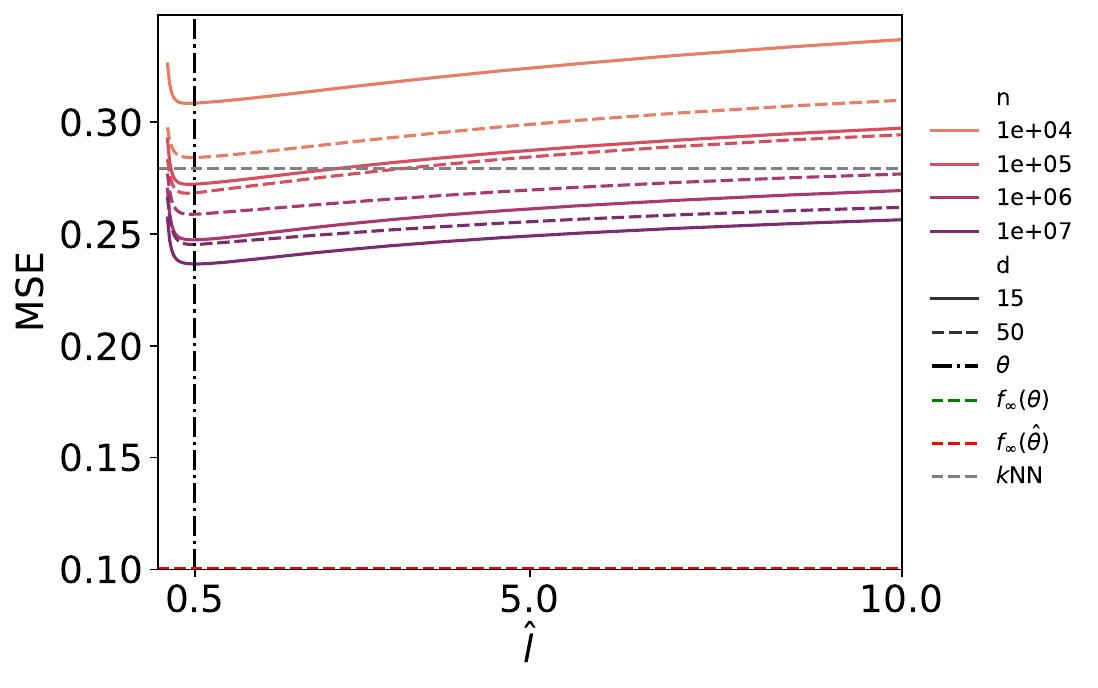}}
\caption{Convergence as a function of lengthscale with matched external parameters (\(c_{Exp}\)-\(\hat{c}_{Exp}\)) for variable \(n,d\). Horizontal grey dashed line represents the best MSE attained by a straight \(k\)-NN prediction (over both dimensions and all \(n\)).}
\label{fig:conv_matched_exp_exp}
\end{center}
\vskip -0.2in
\end{figure}



\subsection{GPnn vs \texorpdfstring{\(k\)}{k}-NN Performance (UCI Datasets)}
\begin{table}[H]
\centering
\setlength{\tabcolsep}{3pt} 
\caption{Results on UCI datasets for \(k\)-NN vs GPnn with RBF and Exponential kernels. Of note: poor performance of \(k\)-NN on the high-dimensional CTslice; performance close to GPnn (RBF) on Road3D. GPnn results taken directly from \cite{GPnn}; \(k\)-NN results obtained via simple modifications to the code used in that paper.}
\label{tab:best_kernels}
\begin{tabular}{llllllll}
\toprule
 &  &  & \multicolumn{5}{l}{\textbf{RMSE}} \\
 &  &  & \(k\)-NN & GPnn (Exp) & GPnn (RBF) \\
Dataset & \(n\) & \(d\) &  &  &  &  &  \\
\midrule
Poletele & 4.6e+03 & 19 & 0.512 ± 0.160 & \bfseries 0.169 ± 0.0076 & 0.195 ± 0.0042  \\
Bike & 1.4e+04 & 13 & 0.752 ± 0.167 & \bfseries 0.565 ± 0.0036 & 0.624 ± 0.0079 \\
Protein & 3.6e+04 & 9 & 0.747 ± 0.163 & \bfseries 0.58 ± 0.0068 & 0.666 ± 0.0014 \\
Ctslice & 4.2e+04 & 378 & 1.66 ± 0.315 & \bfseries 0.123 ± 0.004 & 0.132 ± 0.0006 \\
Road3D & 3.4e+05 & 2 & 0.366 ± 0.091 & \bfseries 0.0976 ± 0.013 & 0.351 ± 0.0014 \\
Song & 4.6e+05 & 90 & 0.882 ± 0.172 & \bfseries 0.776 ± 0.004 & 0.787 ± 0.0045 \\
Houseelectric & 1.6e+06 & 8 & 0.131 ± 0.044 & \bfseries 0.045 ± 0.0003 & 0.0506 ± 0.0007 \\
\bottomrule
\end{tabular}
\end{table}

%% file: appendices/nngp_approx.tex
\section{Nearest Neighbour GP Approximations}
\label{appx:nngps}
Nearest neighbour based GP approximations have a long history in the geo-spatial community, with pioneering work by \cite{vecchia_estimation_1988} on likelihood estimation under local approximation, followed by work such as \cite{Stein2004} who extended this to restricted maximum likelihood (REML) estimation and discussed the differing requirements between parameter estimation and prediction. \citeauthor{Stein2010} further explored the nature of local approximation by considering under what circumstances such approximations are asymptotically effective in terms of a MSE-type assessment, with discussions on what he termed the screening of farther points by the nearest-neighbours (having first raised the point in \cite{Stein2002}). In light of these works, more recent authors have implemented local approximation methods (e.g. \cite{Gramacy2015}) and gone on to develop fully probabilistic hierarchical models (NNGP) in \cite{Datta2016}. Similarly, authors in \cite{Wu2022} implement a stochastic variational approach to the problem, defining VNNGP. We note that the majority of these works focus on the implementation of likelihood approximations for the purposes of parameter estimation in some form. Primarily these sit under the umbrella of empirical Bayes although NNGP adopts a fully Bayesian procedure (in the original paper; an empirical Bayes version is added in \cite{finley_efficient_2018}). In general, the ML community seems to prefer empirical Bayes and \cite{GPnn} focuses on this domain, showing that asymptotically a nearest neighbour based \emph{predictive} mechanism (GPnn) can perform optimally in the sense of pointwise MSE and calibration, \emph{independently} of how the model parameters are estimated and, indeed, under somewhat severe model misspecification. It is interesting to note that both the VNNGP and NNGP models' pointwise-predictive distributions match that given in \cite{GPnn} when the inducing points are selected to match the input data (in the former case) and the conjugate NNGP (using the input data as the reference set) with hyperparameters chosen via empirical Bayes methods is used (in the latter case). 


%% file: refs.bib
@ARTICLE{Datta2016,
  AUTHOR = {Datta, Abhirup and Banerjee, Sudipto and Finley, Andrew O. and Gelfand, Alan E.},
  YEAR = {2016},
  DOI = {10.1080/01621459.2015.1044091},
  EPRINT = {1406.7343},
  EPRINTTYPE = {arXiv},
  ISSN = {1537274X},
  JOURNAL = {Journal of the American Statistical Association},
  NUMBER = {514},
  PAGES = {800--812},
  TITLE = {{Hierarchical Nearest-Neighbor Gaussian Process Models for Large Geostatistical Datasets}},
  VOLUME = {111},
}

@BOOK{Gyorfi2010,
  AUTHOR = {Györfi, László and Kohler, Michael and Krzyżak, Adam and Walk, Harro},
  PUBLISHER = {Springer Series in Statistics},
  YEAR = {2010},
  ISBN = {9781441929983},
  TITLE = {{A Distribution-Free Theory of Nonparametric Regression}},
}

@ARTICLE{GP_approx_review,
  AUTHOR = {Liu, Haitao and Ong, Yew-Soon and Shen, Xiaobo and Cai, Jianfei},
  PUBLISHER = {IEEE},
  YEAR = {2020},
  JOURNAL = {IEEE transactions on neural networks and learning systems},
  NUMBER = {11},
  PAGES = {4405--4423},
  TITLE = {When Gaussian process meets big data: A review of scalable GPs},
  VOLUME = {31},
}

@INPROCEEDINGS{Scholkopf2001,
  AUTHOR = {Schölkopf, Bernhard},
  EDITOR = {Leen, T and Dietterich, T and Tresp, V},
  PUBLISHER = {MIT Press},
  BOOKTITLE = {Advances in Neural Information Processing Systems},
  YEAR = {2000},
  TITLE = {{The Kernel Trick for Distances}},
  VOLUME = {13},
}

@ARTICLE{Stein2004,
  AUTHOR = {Stein, Michael L. and Chi, Zhiyi and Welty, Leah J.},
  YEAR = {2004},
  DOI = {10.1046/j.1369-7412.2003.05512.x},
  ISSN = {13697412},
  JOURNAL = {Journal of the Royal Statistical Society. Series B: Statistical Methodology},
  NUMBER = {2},
  PAGES = {275--296},
  TITLE = {{Approximating likelihoods for large spatial data sets}},
  VOLUME = {66},
}

@InProceedings{Tran2020,
  title = 	 {Sparse within Sparse Gaussian Processes using Neighbor Information},
  author =       {Tran, Gia-Lac and Milios, Dimitrios and Michiardi, Pietro and Filippone, Maurizio},
  booktitle = 	 {Proceedings of the 38th International Conference on Machine Learning},
  pages = 	 {10369--10378},
  year = 	 {2021},
  editor = 	 {Meila, Marina and Zhang, Tong},
  volume = 	 {139},
  series = 	 {Proceedings of Machine Learning Research},
  month = 	 {7},
  publisher =    {PMLR},
  url = 	 {https://proceedings.mlr.press/v139/tran21a.html},
}

@BOOK{GP_book,
  AUTHOR = {Williams, Christopher K and Rasmussen, Carl Edward},
  PUBLISHER = {MIT press Cambridge, MA},
  YEAR = {2006},
  TITLE = {Gaussian processes for machine learning},
  VOLUME = {2},
}

@ARTICLE{Wu2022,
  AUTHOR = {Wu, Luhuan and Pleiss, Geoff and Cunningham, John},
  URL = {http://arxiv.org/abs/2202.01694},
  YEAR = {2022},
  EPRINT = {2202.01694},
  EPRINTTYPE = {arXiv},
  JOURNAL = {Proceedings of the 39th International Conference on Machine Learning},
  TITLE = {{Variational Nearest Neighbor Gaussian Process}},
}

@article{vecchia_estimation_1988,
	title = {Estimation and {Model} {Identification} for {Continuous} {Spatial} {Processes}},
	volume = {50},
	issn = {0035-9246},
	url = {https://www.jstor.org/stable/2345768},
	number = {2},
	urldate = {2023-09-22},
	journal = {Journal of the Royal Statistical Society. Series B (Methodological)},
	author = {Vecchia, A. V.},
	year = {1988},
	note = {Publisher: [Royal Statistical Society, Wiley]},
	pages = {297--312},
}

@article{Stein2002,
	title = {The screening effect in kriging},
	volume = {30},
	issn = {00905364},
	doi = {10.1214/aos/1015362194},
	number = {1},
	journal = {Annals of Statistics},
	author = {Stein, Michael L.},
	year = {2002},
	pages = {298--323},
}

@BOOK{Horn1985,
  AUTHOR = {Horn, Roger A. and Johnson, Charles R.},
  BOOKTITLE = {Matrix Analysis},
EDITION = {Second},
  DATE = {1985},
    YEAR = {1985},
  DOI = {10.1017/cbo9780511810817},
    ISBN = {9780521548236},
  TITLE = {{Matrix Analysis}},
PUBLISHER = {Cambridge University Press}
}

@article{Kohler2006,
author = {Kohler, Michael and Krzyzak, Adam and Walk, Harro},
doi = {10.1016/j.jmva.2005.03.006},
issn = {0047259X},
journal = {Journal of Multivariate Analysis},
number = {2},
pages = {311--323},
title = {{Rates of convergence for partitioning and nearest neighbor regression estimates with unbounded data}},
volume = {97},
year = {2006}
}

@article{Tronarp2018a,
author = {Tronarp, Filip and Karvonen, Toni and Sarkka, Simo},
doi = {10.1109/MLSP.2018.8516992},
isbn = {9781538654774},
issn = {21610371},
journal = {IEEE International Workshop on Machine Learning for Signal Processing, MLSP},
number = {2},
title = {{Mixture representation of the mat{\'{e}}rn class with applications in state space approximations and Bayesian quadrature}},
volume = {2018-September},
year = {2018}
}

@article{Gramacy2015,
	title = {Local {Gaussian} {Process} {Approximation} for {Large} {Computer} {Experiments}},
	volume = {24},
	issn = {15372715},
	doi = {10.1080/10618600.2014.914442},
	number = {2},
	journal = {Journal of Computational and Graphical Statistics},
	author = {Gramacy, Robert B. and Apley, Daniel W.},
	year = {2015},
	note = {arXiv: 1303.0383},
	pages = {561--578},
}

@misc{finley_efficient_2018,
	title = {Efficient algorithms for {Bayesian} {Nearest} {Neighbor} {Gaussian} {Processes}},
	url = {http://arxiv.org/abs/1702.00434},
	urldate = {2023-09-28},
	publisher = {arXiv},
	author = {Finley, Andrew O. and Datta, Abhirup and Cook, Bruce C. and Morton, Douglas C. and Andersen, Hans E. and Banerjee, Sudipto},
	month = mar,
	year = {2018},
	note = {arXiv:1702.00434 [stat]},
}

@inproceedings{vakili_improved_2022,
	title = {Improved {Convergence} {Rates} for {Sparse} {Approximation} {Methods} in {Kernel}-{Based} {Learning}},
	url = {https://proceedings.mlr.press/v162/vakili22a.html},
	language = {en},
	urldate = {2023-06-27},
	booktitle = {Proceedings of the 39th {International} {Conference} on {Machine} {Learning}},
	publisher = {PMLR},
	author = {Vakili, Sattar and Scarlett, Jonathan and Shiu, Da-Shan and Bernacchia, Alberto},
	month = jun,
	year = {2022},
	note = {ISSN: 2640-3498},
	pages = {21960--21983},
}

@inproceedings{GPnn,
      author={Robert Allison and Anthony Stephenson and Samuel F and Edward Pyzer-Knapp},
 booktitle = {Advances in Neural Information Processing Systems},
 publisher = {Curran Associates, Inc.},
      title={Leveraging Locality and Robustness to Achieve Massively Scalable Gaussian Process Regression}, 
 volume = {36},
 year = {2023},
eprint={2306.14731},
      archivePrefix={arXiv},
      primaryClass={stat.ML}
}

@misc{NIST:DLMF,
         key = "{\relax DLMF}",
       title = "{\it NIST Digital Library of Mathematical Functions}",
howpublished = "\url{https://dlmf.nist.gov/}, Release 1.1.11 of 2023-09-15",
         url = "https://dlmf.nist.gov/",
        note = "F.~W.~J. Olver, A.~B. {Olde Daalhuis}, D.~W. Lozier, B.~I. Schneider,
                R.~F. Boisvert, C.~W. Clark, B.~R. Miller, B.~V. Saunders,
                H.~S. Cohl, and M.~A. McClain, eds."}

@article{Stein2010,
	title = {2010 {Rietz} lecture when does the screening effect hold?},
	volume = {39},
	issn = {00905364},
	doi = {10.1214/11-AOS909},
	number = {6},
	journal = {Annals of Statistics},
	author = {Stein, Michael L.},
	year = {2011},
	note = {arXiv: 1203.1801v1},
	pages = {2795--2819},
}

@article{burt_convergence_2020,
	title = {Convergence of {Sparse} {Variational} {Inference} in {Gaussian} {Processes} {Regression}},
	volume = {21},
	language = {en},
	journal = {Journal of Machine Learning Research},
	author = {Burt, David R and Rasmussen, Carl Edward},
	year = {2020},
	pages = {1--63},
}

@article{van_der_vaart_information_2011,
	title = {Information {Rates} of {Nonparametric} {Gaussian} {Process} {Methods}},
	volume = {12},
	url = {https://jmlr.org/papers/volume12/vandervaart11a/vandervaart11a.pdf},
	urldate = {2023-07-24},
	journal = {Journal of Machine Learning Research},
	author = {van der Vaart, Aad},
	year = {2011},
}

@book{Tsybakov2008,
author = {Tsybakov, Alexandre B.},
title = {Introduction to Nonparametric Estimation},
year = {2008},
isbn = {0387790519},
publisher = {Springer Publishing Company, Incorporated},
edition = {1st}
}

@article{Nieman2022,
  author  = {Dennis Nieman and Botond Szabo and Harry van Zanten},
  title   = {Contraction rates for sparse variational approximations in Gaussian process regression},
  journal = {Journal of Machine Learning Research},
  year    = {2022},
  volume  = {23},
  number  = {205},
  pages   = {1--26},
  url     = {http://jmlr.org/papers/v23/21-1128.html}
}

@article{sobolev_spaces_2012,
title = {Hitchhiker's guide to the fractional Sobolev spaces},
journal = {Bulletin des Sciences Mathématiques},
volume = {136},
number = {5},
pages = {521-573},
year = {2012},
issn = {0007-4497},
doi = {https://doi.org/10.1016/j.bulsci.2011.12.004},
url = {https://www.sciencedirect.com/science/article/pii/S0007449711001254},
author = {Eleonora {Di Nezza} and Giampiero Palatucci and Enrico Valdinoci}
}

@article{Geoga2021,
	title = {Fitting {Mat}{\textbackslash}'ern {Smoothness} {Parameters} {Using} {Automatic} {Differentiation}},
	url = {http://arxiv.org/abs/2201.00090},
	author = {Geoga, Christopher J. and Marin, Oana and Schanen, Michel and Stein, Michael L.},
	year = {2021},
	note = {arXiv: 2201.00090},
}

@InCollection{neumann-encyc-maths,
author       =  {{European Mathematical Society}},
title        = {{Neumann series}},
booktitle    =  {{Encyclopedia of Mathematics}},
howpublished =  {\url{http://encyclopediaofmath.org/index.php?title=Neumann_series&oldid=55168}},
year         =  {2024},
publisher    =  {Springer}
}
